\documentclass[12pt,oneside]{article}

\usepackage{setspace}

\usepackage{amssymb}   

\usepackage{amsmath}

\usepackage{mathrsfs}

\usepackage{stmaryrd}

\usepackage{amsthm}
\usepackage{newlfont}
\usepackage{amscd}
\usepackage{mathtools}
\usepackage{bm}
\usepackage{tikz-cd}
\usepackage{graphicx}

\usepackage{hyperref}

\usepackage{enumitem}

\usepackage[all]{xy}

\usepackage{textcomp}

\usepackage{amsbsy}

\newtheorem{thm}{Theorem}[section]

\newtheorem{thm-defn}[thm]{Theorem/Definition}
\newtheorem{lem}[thm]{Lemma}
\newtheorem{prop}[thm]{Proposition}
\newtheorem{cor}[thm]{Corollary}

\theoremstyle{definition}
\newtheorem{defn}[thm]{Definition}

\theoremstyle{remark}

\numberwithin{equation}{section}

\newcommand{\sD}{\mathscr{D}}
\newcommand{\sL}{\mathscr{L}}
\newcommand{\sR}{\mathscr{R}}
\newcommand{\bF}{\mathbf{F}}
\newcommand{\bQ}{\mathbf{Q}}
\newcommand{\bZ}{\mathbf{Z}}
\newcommand{\fa}{\mathfrak{a}}

\newcommand{\fm}{\mathfrak{m}}

\newcommand{\fu}{\mathfrak{u}}

\newcommand{\fM}{\mathfrak{M}}

\newcommand{\fS}{\mathfrak{S}}
\newcommand{\cC}{\mathcal{C}}
\newcommand{\cE}{\mathcal{E}}
\newcommand{\cG}{\mathcal{G}}
\newcommand{\cH}{\mathcal{H}}

\newcommand{\cN}{\mathcal{N}}
\newcommand{\cO}{\mathcal{O}}

\newcommand{\p}{\varphi}

\newcommand{\arr}{\rightarrow}

\newcommand{\T}{\text}
\newcommand{\R}{\mathrm}

\begin{document}

\pagenumbering{arabic}

\title{Potentially semi-stable deformations of specified Hodge-Tate type and Galois type}
\author{Yong Suk Moon}
\date{}

\maketitle

\begin{abstract}
Let $k$ be a perfect field of characteristic $p > 2$, and let $K$ be a finite totally ramified extension of $W(k)[\frac{1}{p}]$. We prove that the locus of potentially semi-stable $\R{Gal}(\bar{K}/K)$-representations of a given Hodge-Tate type and Galois type is a closed subspace of the universal deformation ring, generalizing the result of Kisin (2007) where $k$ is assumed to be finite. 
\end{abstract}

\tableofcontents

\section{Introduction} \label{sec:1}

Let $k$ be a perfect field of characteristic $p > 2$, and let $W(k)$ be its ring of Witt vectors. Write $K_0 = \R{Frac}(W(k))$, and let $K/K_0$ be a finite totally ramified extension. We fix an algebraic closure $\bar{K}$ of $K$, and let $\cG_K \coloneqq \R{Gal}(\bar{K}/K)$ be the absolute Galois group of $K$. 

Let $E/\bQ_p$ be a finite extension with residue field $\bF$, and let $V_0$ be a finite dimensional $\bF$-representation of $\cG_K$. Denote by $\cC$ the category of local topological $\cO_E$-algebras $A$ such that the natural map $\cO_E\rightarrow A/\fm_A$ is surjective and the map from $A$ to the projective limit of its discrete artinian quotients is a topological isomorphism. If $V_0$ is absolutely irreducible, then there exists a universal deformation ring $R \in \cC$ with a deformation $V_R$ which parametrizes the isomorphism classes of deformations of $V_0$ (\cite{smit}). Note that $R$ is not necessarily noetherian in general when $k$ is not finite.

In this paper, we study the geometry of the locus of potentially semi-stable representations with a specified Hodge-Tate type $\bm{v}$ and Galois type $\tau$. We show that such a locus cuts out a closed subspace in the following sense:\\

\noindent\textbf{Theorem A}. \label{thm:1} \emph{There exists a closed ideal $\fa_{\bm{v}, \tau} \subset R$ such that the following holds: for any finite flat $\cO_E$-algebra $A$ and a continuous $\cO_E$-algebra homomorphism $f: R \rightarrow A$ (where we equip $A$ with the $(p)$-adic topology), the induced representation $A[\frac{1}{p}]\otimes_{f, R}V_R$ is potentially semi-stable of Hodge-Tate type $\bm{v}$ and Galois type $\tau$ if and only if $f$ factors through the quotient $R/\fa_{\bm{v}, \tau}$.	}\\

When the residue field $k$ is finite, Kisin proved the corresponding result in \cite[Theorem 2.7.6]{kisin-semistabledeformation}. One of the main steps in \cite{kisin-semistabledeformation} is the construction of the projective scheme which parametrizes representations of $E(u)$-height $\leq r$ for a fixed positive integer $r$ (cf. \cite[Section 1.2]{kisin-semistabledeformation}). It is obtained as a closed subscheme of the affine Grassmannian for the restriction of scalars $\text{Res}_{W(k)/\bZ_p}\text{GL}_d$. But this construction does not make sense in general when $k$ is infinite. The main difficulty is that we do not know how to analyze whether the restriction of scalars  $\text{Res}_{W(k)/\bZ_p}$ for a non-affine scheme over $W(k)$ is representable by an Ind-scheme when $k$ is infinite, even for simple examples such as $\mathbb{P}^1_{W(k)}$. 

Another approach to studying the locus cut out by certain $p$-adic Hodge theoretic conditions, motivated by Fontaine's conjecture in \cite{fontaine-deforming}, is to analyze torsion representations given as the subquotients of Galois stable lattices satisfying the given conditions. For semi-stable (or crystalline) representations having Hodge-Tate weights in $[0, r]$, this is carried out by Liu in \cite{liu-fontaineconjecture}. And in the case $k$ is finite, Liu proved the corresponding result for semi-stable representations of a given Hodge-Tate type in \cite{liu-filtration-torsion}.

We use the functor given in \cite{liu-filteredmodule} from the category of representations semi-stable over a totally ramified Galois extension $K'/K$ to the category of lattices in filtered modules equipped with Frobenius, monodromy, and $\text{Gal}(K'/K)$-action, in order to study the refined structure of a Hodge-Tate type and Galois type of torsion representations. We first study semi-stable representations of a given Hodge-Tate type and show that the locus of such representations is $p$-\textit{adically closed} (cf. Theorem \ref{thm:3.3}). This generalizes the corresponding result in \cite{liu-filtration-torsion} to the case $k$ is not necessarily finite. The proof in \cite{liu-filtration-torsion} is based on reducing to the situation when the coefficient field $E$ contains the Galois closure of $K$, thereby requiring $k$ to be finite. We remove such a restriction using a different argument. 

Then, we study potentially semi-stable representations of a given Galois type, and prove that such representations cut out a $p$-\textit{adically closed} locus (cf. Theorem \ref{thm:3.17}).

\section*{Acknowledgments}

I would like to express my sincere gratitude to Mark Kisin for suggesting me to work on this topic and making many helpful comments. This paper is based on a part of author's Ph.D. thesis under his supervision. I also wish to thank Brian Conrad and Tong Liu for helpful discussions on this work. I thank the referee for a careful reading of the paper and making helpful suggestions to improve it. I would like to thank the department of Mathematics at Harvard University and Purdue University for their cordial environment. This work was partly supported by Samsung Scholarship Foundation, South Korea.

\section{Torsion Representation and Construction of $M_{\mathrm{st}}$} \label{sec:2}

We keep the notations as in the introduction. Let $K' \subset \bar{K}$ be a finite totally ramified Galois extension of $K$. In this section, we will first explain the construction of the functor given in \cite{liu-filteredmodule} from the category of representations semi-stable over $K'$ to the category of lattices in filtered modules equipped with Frobenius, monodromy, and $\text{Gal}(K'/K)$-action. Then, we will explain the result proved in \cite{liu-filteredmodule} and \cite{liu-filtration-torsion} that one can associate a Hodge-Tate type and Galois type to a torsion representation up to some constant depending only on $K'$.

\subsection{Potentially Semi-stable Representation and Filtered $(\p, N, \Gamma)$-module} \label{sec:2.1}

For a $\cG_K$-representation $V$ over $\bQ_p$, we say $V$ is \textit{potentially semi-stable} if there exists a finite extension $L \subset \bar{K}$ of $K$ such that $V$ restricted to $\cG_L \coloneqq \text{Gal}(\bar{K}/L)$ is semi-stable. This means precisely that $\text{dim}_{\bQ_p} V = \text{dim}_{L_0}(B_{\mathrm{st}}\otimes_{\bQ_p}V^{\vee})^{\cG_L}$ where $L_0$ is the maximal unramified subextension of $L/K_0$.

Let $e' = [K':K_0]$. We fix a uniformizer $\pi$ of $K'$, and let $F(u)$ be the Eisenstein polynomial for $\pi$ over $K_0$. Denote by $\text{Rep}_{\bQ_p}^{\mathrm{pst}, K'}$ the category of $\cG_K$-representations over $\bQ_p$ which become semi-stable over $K'$ (i.e., semi-stable as $\text{Gal}(\bar{K}/K')$-representations). Let $\Gamma = \text{Gal}(K'/K)$ and $\cG_{K'} = \text{Gal}(\bar{K}/K')$. Note that $K_0$ is equipped with the natural Frobenius endomorphism $\p$.

We consider the category of filtered $(\p, N, \Gamma)$-modules whose objects are finite dimensional $K_0$-vector spaces $D$ equipped with:

\begin{itemize}
\item a Frobenius semi-linear injection $\p: D \arr D$,
\item $W(k)$-linear map $N: D \arr D$ such that $N\p = p\p N$,
\item decreasing filtration $\text{Fil}^iD_{K'}$ on $D_{K'} \coloneqq K'\otimes_{K_0}D$ by $K'$-sub-vector spaces such that $\text{Fil}^iD_{K'} = D_{K'}$ for $i \ll 0$	and $\text{Fil}^iD_{K'} = 0$ for $i \gg 0$, and
\item $K_0$-linear action by $\Gamma$ on $D$ which commutes with $\p$ and $N$. If we extend $\Gamma$-action semi-linearly to $D_{K'}$, then for any $\gamma \in \Gamma$, $\gamma(\text{Fil}^i D_{K'}) \subset \text{Fil}^i D_{K'}$. 
\end{itemize}
 
Morphisms between filtered $(\p, N, \Gamma)$-modules are $K_0$-linear maps compatible with all structures. The functor $D_{\mathrm{st}}^{K'}: V \mapsto (B_{\mathrm{st}}\otimes_{\bQ_p} V^{\vee})^{\cG_{K'}}$ is an equivalence between $\text{Rep}_{\bQ_p}^{\mathrm{pst}, K'}$ and the category of \textit{weakly admissible} filtered $(\p, N, \Gamma)$-modules (cf. \cite{colmez-fontaine}, \cite{fontaine-representations}).

We define an integral structure of a filtered $(\p, N, \Gamma)$-module.

\begin{defn} \label{defn:2.1}
Let $D$ be a filtered $(\p, N, \Gamma)$-module. A \textit{lattice} $M$ in $D$ is a finite free $W(k)$-submodule of $D$ such that $M[\frac{1}{p}] \coloneqq M\otimes_{\bZ_p}\bQ_p \cong D$, and $\p(M) \subset M$, $N(M) \subset M$, and $\gamma(M) \subset M$ for all $\gamma \in \Gamma$. For a lattice $M \subset D$, we equip $M_{K'} \coloneqq \cO_{K'}\otimes_{W(k)} M$ with the natural filtration by $\cO_{K'}$-submodules, given by $\R{Fil}^iM_{K'} = M_{K'} \cap \R{Fil}^i D_{K'}$. If $M_1, M_2$ are lattices in filtered $(\p, N, \Gamma)$-modules $D_1, D_2$ respectively, then a morphism $f: M_1 \arr M_2$ is a $W(k)$-linear map such that $f \otimes_{\bZ_p}\bQ_p : D_1 \arr D_2$ is a morphism of filtered $(\p, N, \Gamma)$-modules.  
\end{defn}

Note that for a lattice $M$ in a filtered $(\p, N, \Gamma)$-module, the associated graded $\cO_{K'}$-modules $\R{gr}^iM_{K'} = \R{Fil}^iM_{K'}/\R{Fil}^{i+1}M_{K'}$ is torsion free by the definition of the filtration. 
  
Let $r$ be a positive integer. Denote by $L^r(\p, N, \Gamma)$ the category of lattices in filtered $(\p, N, \Gamma)$-modules $D$ satisfying $\text{Fil}^0 D_{K'} = D_{K'}$ and $\T{Fil}^{r+1}D_{K'} = 0$. Let $\T{Rep}_{\bQ_p}^{\mathrm{pst}, K', r}$ be the full subcategory of $\text{Rep}_{\bQ_p}^{\mathrm{pst}, K'}$ whose objects have Hodge-Tate weights in $[0, r]$, and let $\T{Rep}_{\bZ_p}^{\mathrm{pst}, K', r}$ be the category of $\cG_K$-stable $\bZ_p$-lattices of representations in $\T{Rep}_{\bQ_p}^{\mathrm{pst}, K', r}$. The following theorem is proved in \cite{liu-filteredmodule}:

\begin{thm} \emph{(cf. \cite[Theorem 2.3]{liu-filteredmodule})} \label{thm:2.2}
There exists a faithful contravariant functor $M_{\mathrm{st}}$ from $\emph{Rep}_{\bZ_p}^{\mathrm{pst}, K', r}$ to $L^r(\p, N, \Gamma)$. If we denote by $M_{\mathrm{st}}\otimes_{\bZ_p}\bQ_p$ the functor $M_{\mathrm{st}}$ associated to the isogeny categories, then there exists a natural isomorphism of functors between $M_{\mathrm{st}}\otimes_{\bZ_p}\bQ_p$ and $D_{\mathrm{st}}^{K'}$.	
\end{thm}

\subsection{Construction of $M_{\mathrm{st}}$} \label{sec:2.2}

We now explain briefly the construction in \cite{liu-filteredmodule} of the functor $M_{\mathrm{st}}$ in Theorem \ref{thm:2.2}. We first recall the definitions of period rings in $p$-adic Hodge theory. 

Let $S'$ be the $p$-adic completion of the divided power-envelope of $\fS = W(k)[\![u]\!]$ with respect to the ideal $(F(u))$. Denote $S'_{K_0} \coloneqq S'[\frac{1}{p}]$. Let $C_p$ be the $p$-adic completion of $\bar{K}$, and let $\cO_{C_p}$ be its ring of integers. We define $\cO_{C_p}^{\flat} \coloneqq \displaystyle\varprojlim_{x \mapsto x^p} \cO_{C_p}/p$. By the universal property of the ring of Witt vectors $W(\cO_{C_p}^{\flat})$, there exists a unique surjection $\theta: W(\cO_{C_p}^{\flat}) \arr \cO_{C_p}$, which lifts the projection $\cO_{C_p}^{\flat} \arr \cO_{C_p}/p$ onto the first factor of the inverse limit. We denote by $B^+_{\mathrm{dR}}$ the $\T{ker}(\theta)$-adic completion of $W(\cO_{C_p}^{\flat})[\frac{1}{p}]$. Let $A_{\mathrm{cris}}$ be the $p$-adic completion of the divided power-envelope of $W(\cO_{C_p}^{\flat})$ with respect to $\T{ker}(\theta)$. We fix a compatible system of $p^n$-th roots $\pi_n \in \cO_{\bar{K}}$ of $\pi$ for non-negative integers $n$, and let $\underline{\pi} \coloneqq (\pi_n) \in \cO_{C_p}^{\flat}$. We have an embedding $\fS \hookrightarrow W(\cO_{C_p}^{\flat})$ mapping $u$ to $[\underline{\pi}]$, and hence the embeddings $\fS \hookrightarrow S' \hookrightarrow A_{\mathrm{cris}}$ compatible with Frobenius endomorphisms. Let $B^+_{\mathrm{cris}} = A_{\mathrm{cris}}[\frac{1}{p}]$. Let $\fu = \T{log}[\underline{\pi}]$, and $B^+_{\mathrm{st}} = B^+_{\mathrm{cris}}[\fu]$. We also fix a compatible system of primitive $p^n$-th roots of unity $\zeta_{p^n} \in \cO_{\bar{K}}$ for non-negative integers $n$, and let $\underline{\epsilon} \coloneqq (\zeta_{p^n}) \in \cO_{C_p}^{\flat}$. Let $t = \T{log}[\underline{\epsilon}] \in B^+_{\mathrm{dR}}$. Note that we also have $t \in A_{\mathrm{cris}}$. Let $B_{\mathrm{dR}} = B^+_{\mathrm{dR}}[\frac{1}{t}], ~B_{\mathrm{cris}} = B^+_{\mathrm{cris}}[\frac{1}{t}]$, and $B_{\mathrm{st}} = B^+_{\mathrm{st}}[\frac{1}{t}]$.

We denote by $\cO_{\cE}$ the $p$-adic completion of $\fS[\frac{1}{u}]$, and let $\cE = \T{Frac}(\cO_{\cE})$. Let $\hat{\cE}^{\mathrm{ur}}$ be the $p$-adic completion of the maximal unramified subextension of $\cE$ in $W(\T{Frac}(\cO_{C_p}^{\flat}))[\frac{1}{p}]$, and $\cO_{\hat{\cE}^{\mathrm{ur}}}$ its ring of integers. We let $\fS^{\mathrm{ur}} = \cO_{\hat{\cE}^{\mathrm{ur}}} \cap W(\cO_{C_p}^{\flat})$.

We let $K'_{\infty} = \displaystyle\bigcup_{n=1}^{\infty} K'(\pi_n)$ and $K'_{p^{\infty}} = \displaystyle\bigcup_{n=1}^{\infty} K'(\zeta_{p^n})$. Let $K'_c = K'_{\infty}K'_{p^{\infty}}$, which is the Galois closure of $K'_{\infty}$ over $K'$. Let $\hat{\cG} = \T{Gal}(K'_c/K'), ~\cG_{\infty} = \T{Gal}(\bar{K}/K'_{\infty})$, and $\cH_{K'} = \T{Gal}(K'_c/K'_{\infty})$. Write 
\[
t^{\{i\}} = \frac{t^i}{p^{q(i)}q(i)!}
\] 
where $q(i)$ is defined by $i = q(i)(p-1)+r(i)$ with $0 \leq r(i) < p-1$. We define
\[
\sR_{K_0} \coloneqq \{\sum_{i=0}^{\infty} a_{i}t^{\{i\}} ~~|~~ a_{i} \in S'_{K_{0}},~~ a_{i} \rightarrow 0 \hspace{0.5em} p\T{-adically as} \hspace{0.5em} i \rightarrow \infty\}.
\]
We have a natural map $\nu: W(\cO_{C_p}^{\flat}) \arr W(\bar{k})$ induced by the projection $\cO_{C_p}^{\flat} \arr \bar{k}$, which can be seen to extend uniquely to $\nu: B^+_{\mathrm{cris}} \arr W(\bar{k})[\frac{1}{p}]$. For any subring $A \subset B^+_{\mathrm{cris}}$, write $I_+A \coloneqq A \cap \T{ker}(\nu)$. We have $I_+\fS = u\fS$ and 
\[
I_+S' = \{\sum_{i=1}^{\infty} \frac{b_{i}}{\lfloor \frac{i}{e'} \rfloor!} u^i ~|~ b_i \in W(k), ~b_i \arr 0  ~p\T{-adically as} ~i \arr \infty\}. 
\]
Define $\hat{\sR} = W(\cO_{C_p}^{\flat}) \cap \sR_{K_0}$ and $I_+ = I_+\hat{\sR}$. The following lemma is proved in \cite{liu-semistable-lattice}.

\begin{lem} \emph{(\cite[Lemma 2.2.1]{liu-semistable-lattice})} \label{lem:2.3}
\begin{enumerate}
\item $\hat{\sR}$ (resp. $\sR_{K_0}$) is a $\p$-stable $\fS$-algebra as a subring in $W(\cO_{C_p}^{\flat})$ (resp. $B^+_{\mathrm{cris}}$).
\item $\hat{\sR}$ and $I_+$ (resp. $\sR_{K_0}$ and $I_+\sR_{K_0}$) are $\cG_{K'}$-stable. The $\cG_{K'}$-actions on $\hat{\sR}$ and $I_+$ (resp. $\sR_{K_0}$ and $I_+\sR_{K_0}$) factor through $\hat{\cG}$.
\item $\sR_{K_0}/I_+\sR_{K_0} \cong K_0$ and $\hat{\sR}/I_+ \cong S'/I_+S' \cong \fS/u\fS \cong W(k)$.	
\end{enumerate}  
\end{lem}

Let $r$ be a positive integer. A \textit{Kisin module of height} $r$ is a pair $(\fM, \p_{\fM})$ where $\fM$ is a finite free $\fS$-module, and $\p_{\fM}: \fM \arr \fM$ is a $\p$-semi-linear map such that the cokernel of the induced map $1\otimes \p_{\fM}: \p^*(\fM) \arr \fM$ is killed by $F(u)^r$. A morphism between two Kisin modules $\fM_1, \fM_2$ is a morphism as $\fS$-modules compatible with $\p_{\fM_i}$. Let $\T{Mod}_{\fS}^r(\p)$ denote the category of Kisin modules of height $r$. For $(\fM, \p_{\fM}) \in \T{Mod}_{\fS}^r(\p)$, we write $\hat{\fM} = \hat{\sR}\otimes_{\p, \fS}\fM$. The Frobenius $\p_{\fM}$ on $\fM$  naturally extends to $\hat{\fM}$ by $\p_{\hat{\fM}}(a\otimes m) = \p_{\hat{\sR}}(a)\otimes \p_{\fM}(m)$.

\begin{defn} \label{defn:2.4}
A $(\p, \hat{\cG})$-\textit{module of height} $r$ is a triple $(\fM, \p_{\fM}, \hat{\cG}_{\fM})$ satisfying the following:
\begin{itemize}
\item $(\fM, \p_{\fM})$ is  Kisin module of height $r$.
\item $\hat{\cG}_{\fM}$ denotes a $\hat{\sR}$-semi-linear $\hat{\cG}$-action on $\hat{\fM}$ which commutes with $\p_{\hat{\fM}}$ and induces a trivial action on $\hat{\fM}/I_+ \hat{\fM}$.
\item Considering $\fM$ as a $\p(\fS)$-submodule of $\hat{\fM}$, we have $\fM \subset \hat{\fM}^{\cH_{K'}}$.
\end{itemize}	
\end{defn}

A morphism between two $(\p, \hat{\cG})$-modules $\fM_1, \fM_2$ of height $r$ is a morphism in $\T{Mod}_{\fS}^r(\p)$ which commutes with the $\hat{\cG}$-actions. We denote by $\T{Mod}_{\fS}^r(\p, \hat{\cG})$ the category of $(\p, \hat{\cG})$-modules of height $r$. For $\hat{\fM} \in \T{Mod}_{\fS}^r(\p, \hat{\cG})$, we associate a $\bZ_p[\cG_{K'}]$-module $\hat{T}^{\vee}(\hat{\fM}) \coloneqq \T{Hom}_{\hat{\sR}, \p}(\hat{\fM}, W({\cO_{C_p}^{\flat}}))$ with $\cG_{K'}$-action given by $g(f)(x) = g(f(g^{-1}(x)))$ for $g \in \cG_{K'}, ~f \in \hat{T}^{\vee}(\hat{\fM})$. Here, $\cG_{K'}$-action on $\hat{\fM}$ is given by $\hat{\cG}$-action on $\hat{\fM}$. Moreover, for $\fM \in \T{Mod}_{\fS}^r(\p)$, we associate a $\bZ_p[\cG_{\infty}]$-module $T_{\fS}^{\vee}(\fM) \coloneqq \T{Hom}_{\fS, \p}(\fM, \fS^{\mathrm{ur}})$ similarly. The main result proved in \cite{liu-semistable-lattice} is the following.

\begin{thm} \emph{(cf. \cite[Theorem 2.3.1, Proposition 3.1.3]{liu-semistable-lattice})} \label{thm:2.5}
\begin{enumerate}
\item $\hat{T}^{\vee}$ induces an anti-equivalence between $\emph{Mod}_{\fS}^r(\p, \hat{\cG})$ and the category of $\cG_{K'}$-stable $\bZ_p$-lattices in semi-stable representations of $\cG_{K'}$ having Hodge-Tate weights in $[0, r]$.
\item $\hat{T}^{\vee}$ induces a natural $W({\cO_{C_p}^{\flat}})$-linear injection
\[
\hat{\iota}: W({\cO_{C_p}^{\flat}})\otimes_{\hat{\sR}}\hat{\fM} \arr W({\cO_{C_p}^{\flat}})\otimes_{\bZ_p}\hat{T}(\hat{\fM})	
\]
such that $\hat{\iota}$ is compatible with Frobenius maps and $\cG_{K'}$-actions on both sides. Here, $\hat{T}(\hat{\fM}) \coloneqq \R{Hom}_{\bZ_p}(\hat{T}^{\vee}(\hat{\fM}), \bZ_p)$.
\item There exists a natural isomorphism $T_{\fS}^{\vee}(\fM) \stackrel{\cong}{\arr} \hat{T}^{\vee}(\hat{\fM})$ of $\bZ_p[\cG_{\infty}]$-modules.
\end{enumerate}
\end{thm}

To construct the functor $M_{\mathrm{st}}$, we establish a connection between $(\p, \hat{\cG})$-modules and filtered $(\p, N, \Gamma)$-modules. Let $V \in \T{Rep}_{\bQ_p}^{\mathrm{pst}, K', r}$, and let $T \subset V$ be a $\cG_K$-stable $\bZ_p$-lattice. By Theorem \ref{thm:2.5}, there exists a unique $\fM \in \T{Mod}_{\fS}^r(\p, \hat{\cG})$ such that $\hat{T}^{\vee}(\hat{\fM}) = T$ as $\bZ_p[\cG_{K'}]$-modules. Let $\sD \coloneqq S'_{K_0}\otimes_{\p, \fS}\fM$ equipped with the Frobenius endomorphism given by $\p_{\sD} = \p_{S'_{K_0}}\otimes\p_{\fM}$. Let $D = \sD/(I_+S'_{K_0})\sD$, which is a finite $K_0$-vector space equipped with the Frobenius induced from $\p_{\sD}$. By \cite[Proposition 6.2.1.1]{breuil-representations}, there exists a unique section $s: D \arr \sD$ compatible with the Frobenius morphisms on both sides. Thus, $\sD = S'_{K_0}\otimes_{K_0} D$ if we identify $D$ with $s(D)$. So $B^+_{\mathrm{cris}}\otimes_{\hat{\sR}}\hat{\fM} \cong B^+_{\mathrm{cris}}\otimes_{K_0}D$, and the map $\hat{\iota}$ given in Theorem \ref{thm:2.5} (2) induces a natural injection $D \hookrightarrow B^+_{\mathrm{cris}}\otimes_{\bZ_p}T^{\vee}$ where $T^{\vee} \coloneqq \R{Hom}_{\bZ_p}(T, \bZ_p)$.

On the other hand, the functor $D_{\mathrm{st}}^{K'}$ induces an injection
\[
\iota: B^+_{\mathrm{st}}\otimes_{K_0} D_{\mathrm{st}}^{K'}(V) \arr B^+_{\mathrm{st}}\otimes_{\bQ_p} V^{\vee}
\]
such that $\iota$ is compatible with Frobenius, monodromy, filtration, and $\cG_{K'}$-action on both sides. The following is proved in \cite{liu-filteredmodule}.

\begin{prop} \emph{(cf. \cite[Proposition 2.6, Corollary 2.7, 2.8]{liu-filteredmodule})} \label{prop:2.6}
There exists a unique $K_0$-linear isomorphism $i: D_{\mathrm{st}}^{K'}(V) \arr D$ such that $i$ is compatible with the Frobenius morphisms on both sides and makes the following diagram commutative:\\

\begin{center}
\includegraphics[scale=1]{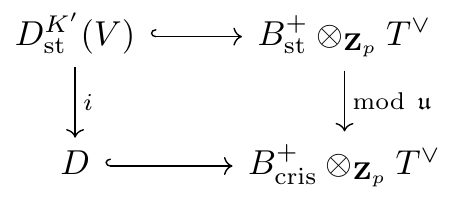}
\end{center}

\noindent Furthermore, such $i$ is functorial. 	
\end{prop}

Note that 
\[
\fM/u\fM \cong \p^*(\fM)/u\p^*(\fM) \subset \sD/I_+S'_{K_0}\sD = D.
\]
We set $M_{\mathrm{st}}(T) \subset D_{\mathrm{st}}^{K'}(V)$ to be the inverse image of $\p^*(\fM)/u\p^*(\fM)$ under the isomorphism $i: D_{\mathrm{st}}^{K'}(V) \arr D$ given in Proposition \ref{prop:2.6}. $M_{\mathrm{st}}(T)$ is a finite free $W(k)$-lattice in $D_{\mathrm{st}}^{K'}(V)$ stable under Frobenius. Furthermore, it is proved in \cite[Corollary 2.12, Proposition 2.15]{liu-filteredmodule} that $M_{\mathrm{st}}(T)$ is stable under $\cG_K$-action and $N$ on $D_{\mathrm{st}}^{K'}(V)$. Thus, $M_{\mathrm{st}}(T)$ is a lattice of the filtered $(\p, N, \Gamma)$-module $D_{\mathrm{st}}^{K'}(V)$. And the association $M_{\R{st}}(\cdot)$ is a contravariant functor from $\R{Rep}_{\bZ_p}^{\mathrm{pst}, K', r}$ to $L^r(\p, N, \Gamma)$ since the isomorphism $i$ in Proposition \ref{prop:2.6} is functorial.

\subsection{Potentially Semi-stable Torsion Representations} \label{sec:2.3}

We now associate torsion filtered $(\p, N, \Gamma)$-modules to potentially semi-stable torsion representations. Denote by $\T{Rep}_{\mathrm{tor}}^{\mathrm{pst}, K', r}$ the category of torsion representations $L$ semi-stable over $K'$ and of height $r$, in a sense that there exist lattices $\sL_1, \sL_2 \in \T{Rep}_{\bZ_p}^{\mathrm{pst}, K', r}$ with a $\cG_K$-equivariant injection $j: \sL_1 \hookrightarrow \sL_2$ such that $L \cong \sL_2/j(\sL_1)$ as $\bZ_p[\cG_K]$-modules, and $L$ is killed by some power of $p$. Morphisms between two torsion representations in $\T{Rep}_{\mathrm{tor}}^{\mathrm{pst}, K', r}$ are morphisms of $\bZ_p[\cG_K]$-modules. We call such $(\sL_1, \sL_2, j)$ a \textit{lift} of $L$. We will sometimes denote simply by $j$ a lift of $L$. Note that a lift of $L \in \T{Rep}_{\mathrm{tor}}^{\mathrm{pst}, K', r}$ is not unique. Let $L, L' \in \T{Rep}_{\mathrm{tor}}^{\mathrm{pst}, K', r}$ with lifts $(\sL_1, \sL_2, j), (\sL_1', \sL_2', j')$ respectively. If $f: L \arr L'$ is a morphism in $\T{Rep}_{\mathrm{tor}}^{\mathrm{pst}, K', r}$, we say a morphism $\tilde{f}: \sL_2 \arr \sL_2'$ in $\T{Rep}_{\bZ_p}^{\mathrm{pst}, K', r}$ is a \textit{lift} of $f$ if $\tilde{f}(j(\sL_1)) \subset j'(\sL_1')$ and $\tilde{f}$ induces $f$.

We denote by $M_{\mathrm{tor}}^{\mathrm{fil}, r}(\p, N, \Gamma)$ the category whose objects are finite $W(k)$-modules $M$ killed by some power of $p$ and endowed with the following structures:
\begin{itemize}
\item a Frobenius semilinear morphism $\p: M \arr M$,
\item $W(k)$-linear map $N: M \arr M$ satisfying $N\p= p\p N$,
\item $W(k)$-linear $\Gamma$-action on $M$ which commutes with $\p$ and $N$, and
\item $M_{K'} \coloneqq \cO_{K'}\otimes_{W(k)}M$ has decreasing filtration by $\cO_{K'}$-submodules such that $\T{Fil}^0 M_{K'} = M_{K'}$ and $\T{Fil}^{r+1} M_{K'} = 0$. Also, $\gamma(\T{Fil}^i M_{K'}) \subset \T{Fil}^i M_{K'}$ for any $\gamma \in \Gamma$.	
\end{itemize}
  
Morphisms in $M_{\mathrm{tor}}^{\mathrm{fil}, r}(\p, N, \Gamma)$ are $W(k)$-linear maps compatible with above structures. For $L \in \T{Rep}_{\mathrm{tor}}^{\mathrm{pst}, K', r}$ with a lift $j: \sL_1 \hookrightarrow \sL_2$, we can associate an object $M_{\mathrm{st}, j}(L) \in M_{\mathrm{tor}}^{\mathrm{fil}, r}(\p, N, \Gamma)$ as follows. By Theorem \ref{thm:2.2}, we have the morphism $M_{\mathrm{st}}(j): M_{\mathrm{st}}(\sL_2) \arr M_{\mathrm{st}}(\sL_1)$ in $L^r(\p, N, \Gamma)$ corresponding to $j$, and $M_{\mathrm{st}}(j)$ is injective by \cite[Corollary 3.8]{liu-filteredmodule}. We set $M_{\mathrm{st}, j}(L) = M_{\mathrm{st}}(\sL_1)/M_{\mathrm{st}}(j)(M_{\mathrm{st}}(\sL_2))$. Then $M_{\mathrm{st}, j}(L)$ has natural endomorphisms $\p$ and $N$, and $\Gamma$-action induced from $M_{\mathrm{st}}(\sL_1)$. Furthermore, tensoring by $\cO_{K'}$ on $M_{\mathrm{st}}(j)$ gives the following exact sequence:
\[
0 \arr \cO_{K'}\otimes_{W(k)} M_{\mathrm{st}}(\sL_2) \arr \cO_{K'}\otimes_{W(k)} M_{\mathrm{st}}(\sL_1) \stackrel{q}{\arr} \cO_{K'}\otimes_{W(k)}M_{\mathrm{st}, j}(L) \arr 0.
\] 
We define the filtration on $M_{\mathrm{st}, j}(L)_{K'}$ by $\T{Fil}^i M_{\mathrm{st}, j}(L)_{K'} \coloneqq q(\T{Fil}^i M_{\mathrm{st}}(\sL_1)_{K'})$. This gives $M_{\mathrm{st}, j}(L)$ a structure as an object in $M_{\mathrm{tor}}^{\mathrm{fil}, r}(\p, N, \Gamma)$. By the snake lemma, we further have the following exact sequence of the associated graded modules:
\[
0 \arr \T{gr}^i(M_{\mathrm{st}}(\sL_2)_{K'}) \arr \T{gr}^i(M_{\mathrm{st}}(\sL_1)_{K'}) \arr \T{gr}^i(M_{\mathrm{st}, j}(L)_{K'}) \arr 0.
\]  
If $f: L \arr L'$ is a morphism in $\T{Rep}_{\mathrm{tor}}^{\mathrm{pst}, K', r}$ with a lift $\tilde{f}: (\sL_1, \sL_2, j) \arr (\sL_1', \sL_2', j')$, then it induces a morphism $M_{\mathrm{st}, \tilde{f}}(f): M_{\mathrm{st}, j'}(L') \arr M_{\mathrm{st}, j}(L)$ in $M_{\mathrm{tor}}^{\mathrm{fil}, r}(\p, N, \Gamma)$.

Note that the above construction depends on the choice of the lift of $L$. However, the following theorem, which can be deduced directly from \cite{liu-filtration-torsion} and \cite{liu-filteredmodule}, shows that the construction depends on lifts only up to a constant.

\begin{thm} \label{thm:2.7}
There exists a constant $c$ depending only on $F(u)$ and $r$ such that the following statement holds: for any morphism $f: L \arr L'$ in $\emph{Rep}_{\mathrm{tor}}^{\mathrm{pst}, K', r}$ with lifts $j, j'$ of $L, L'$ respectively, there exists a morphism $\tilde{h}: M_{\mathrm{st}, j'}(L') \arr M_{\mathrm{st}, j}(L)$ in $M_{\mathrm{tor}}^{\mathrm{fil}, r}(\p, N, \Gamma)$ such that
\begin{itemize}
\item if there exists a morphism of lifts $\tilde{f}: j \arr j'$ which lifts $f$, then $\tilde{h} = p^cM_{\mathrm{st}, \tilde{f}}(f)$,
\item let $f': L' \arr L''$ be a morphism in $\emph{Rep}_{\mathrm{tor}}^{\mathrm{pst}, K', r}$, $j''$ a lift of $L''$, and $\tilde{h}': M_{\mathrm{st}, j''}(L'') \arr M_{\mathrm{st}, j'}(L')$ the morphism in $M_{\mathrm{tor}}^{\mathrm{fil}, r}(\p, N, \Gamma)$ associated to $f', j'$, and $j''$. If there exists a morphism of lifts $\tilde{g}: j \arr j''$ which lifts $f' \circ f$, then $\tilde{h} \circ \tilde{h}' = p^{2c}M_{\mathrm{st}, \tilde{g}}(f' \circ f)$.
\end{itemize}
\end{thm}

\begin{proof}
It follows directly from \cite[Theorem 3.1]{liu-filteredmodule} and \cite[Theorem 2.1.3, Remark 2.1.5]{liu-filtration-torsion}.	
\end{proof}

The following corollary is immediate. 

\begin{cor} \emph{(cf. \cite[Corollary 3.2]{liu-filteredmodule}, \cite[Corollary 2.1.4]{liu-filtration-torsion})} \label{cor:2.8}
With notations as in Theorem \ref{thm:2.7}, assume that $f: L \arr L'$ is an isomorphism with the inverse $f^{-1}: L' \arr L$. Let $\tilde{h}_1: M_{\mathrm{st}, j}(L) \arr M_{\mathrm{st}, j'}(L')$ be the morphism as in Theorem \ref{thm:2.7} associated to $f^{-1}, j$, and $j'$. Then $\tilde{h} \circ \tilde{h}_1 = p^{2c}\mathrm{Id}$ on $M_{\mathrm{st}, j}(L)$ and $\tilde{h}_1 \circ \tilde{h} = p^{2c}\mathrm{Id}$ on $M_{\mathrm{st}, j'}(L')$. Furthermore, the similar statement holds for the induced morphisms on $\R{gr}^i(M_{\R{st}, j}(L)_{K'})$ and $\R{gr}^i(M_{\R{st}, j'}(L')_{K'})$.	
\end{cor}

\subsection{Representation with Coefficient} \label{sec:2.4}

Let $A$ be a $\bZ_p$-algebra, and denote by $\R{Rep}_A^{\R{pst}, K', r}$ the subcategory of $\R{Rep}_{\bZ_p}^{\R{pst}, K', r}$ whose objects are $A$-modules such that $\cG_K$-actions are $A$-linear. Morphisms in $\R{Rep}_A^{\R{pst}, K', r}$ are morphisms of $A[\cG_K]$-modules. Let $\R{Rep}_{\R{tor}, A}^{\R{pst}, K', r}$ be the subcategory of $\R{Rep}_{\R{tor}}^{\R{pst}, K', r}$ whose objects have lifts in $\R{Rep}_A^{\R{pst}, K', r}$, and the morphisms are $A[\cG_K]$-module morphisms. For $L \in \R{Rep}_{\R{tor}, A}^{\R{pst}, K', r}$ having a lift $j: \sL_1 \hookrightarrow \sL_2$ in  $\R{Rep}_A^{\R{pst}, K', r}$, note that $M_{\R{st}}(\sL_1)$ and $M_{\R{st}}(\sL_2)$ are naturally $A\otimes_{\bZ_p}W(k)$-modules, and thus so is $M_{\R{st}, j}(L)$.

\begin{prop} \label{prop:2.9}
Let $f: L \arr L'$ be a morphism in $\R{Rep}_{\R{tor}, A}^{\R{pst}, K', r}$, and let $j$ and $j'$ be lifts in $\R{Rep}_A^{\R{pst}, K', r}$ of $L$ and $L'$ respectively. 	Then, the associated morphism $\tilde{h}: M_{\R{st}, j'}(L') \arr M_{\R{st}, j}(L)$ in  $M_{\mathrm{tor}}^{\mathrm{fil}, r}(\p, N, \Gamma)$ as in Theorem \ref{thm:2.7} is a morphism of $A\otimes_{\bZ_p}W(k)$-modules. 
\end{prop}

\begin{proof}
It follows immediately from \cite[Proposition 3.13]{liu-filteredmodule} and \cite[Lemma 4.2.4]{liu-filtration-torsion}	.
\end{proof}

\section{Hodge-Tate Type and Galois Type} \label{sec:3}

\subsection{Hodge-Tate Type} \label{sec:3.1}

Let $E$ be a finite extension of $\bQ_p$, and let $B$ be a finite $E$-algebra. Let $V_B$ be a free $B$-module of rank $d$ equipped with a continuous $\cG_K$-action. Suppose that as a representation of $\cG_K$, $V_B$ is semi-stable over $K'$, i.e., $V_B \in \R{Rep}_{\bQ_p}^{\R{pst}, K'}$. Then $V_B$ is de Rham over $K$, and we set $D_{\R{dR}}^K(V_B) = (B_{\R{dR}}\otimes_{\bQ_p}V_B^{\vee})^{\cG_K}$. For any $E$-algebra $A$, we write $A_K \coloneqq A\otimes_{\bQ_p}K$.

\begin{lem} \emph{(cf. \cite[Lemma 4.1.2]{liu-filtration-torsion})} \label{lem:3.1}
\begin{enumerate}
\item Let $B'$ be a finite $B$-algebra, and write $V_{B'} = B'\otimes_B V_B$. Then $D_{\R{dR}}^K (V_{B'}) \cong B'\otimes_B D_{\R{dR}}^K (V_B)$, and $\R{gr}^i(D_{\R{dR}}^{K}(V_{B'})) \cong B'\otimes_B \R{gr}^i(D_{\R{dR}}^K (V_{B}))$.
\item $D_{\R{dR}}^K (V_{B})$ is a free $B_K$-module of rank $d$.	
\end{enumerate}
\end{lem}

\begin{proof}
(1) is proved in \cite[Lemma 4.1.2]{liu-filtration-torsion}. For (2), since $D_{\R{dR}}^{K}(V_{B}) = K\otimes_{K_0}D_{\R{st}}^{K'}(V_{B})$	, it suffices to prove that $D_{\R{st}}^{K'}(V_{B})$ is a free $B\otimes_{\bQ_p}K_0$-module of rank $d$. For any finite $B$-algebra $B'$, we can show similarly as in (1) that $D_{\R{st}}^{K'}(V_{B'}) \cong B'\otimes_B D_{\R{st}}^{K'}(V_B)$. Let $B_{\R{red}} = B/\mathcal{N}(B)$ where $\cN(B)$ denotes the nilradical of $B$. $B_{\R{red}}$ is a reduced Artinian ring, so there exists a ring isomorphism $B_{\R{red}} \cong \prod_{j=1}^m E_j$ for some field $E_j$ finite over $E$. $E_j\otimes_{\bQ_p}K_0$ is isomorphic to a finite direct product of fields, so $D_{\R{st}}^{K'}(V_{E_j}) \cong E_j \otimes_{B} D_{\R{st}}^{K'}(V_{B})$ is finite projective as an $E_j\otimes_{\bQ_p}K_0$-module. Note that the Frobenius morphism on $K_0$ extends $E_j$-linearly to $E_j\otimes_{\bQ_p}K_0$, and the extended Frobenius permutes the maximal ideals of $E_j\otimes_{\bQ_p}K_0$ transitively. Therefore, $D_{\R{st}}^{K'}(V_{E_j})$ is a free $E_j\otimes_{\bQ_p}K_0$-module of rank $d$, and $D_{\R{st}}^{K'}(V_{B_{\R{red}}}) = B_{\R{red}}\otimes_B D_{\R{st}}^{K'}(V_B)$ is a free $B_{\R{red}}\otimes_{\bQ_p}K_0$-module of rank $d$.

Let $\{e_1, \ldots, e_d\}$ be a $B_{\R{red}}\otimes_{\bQ_p}K_0$-basis of $D_{\R{st}}^{K'}(V_{B_{\R{red}}})$, and choose a lift $\hat{e}_i \in D_{\R{st}}^{K'}(V_B)$ of $e_i$. By Nakayama's lemma, $\{\hat{e}_1, \ldots, \hat{e}_d\}$ generate $D_{\R{st}}^{K'}(V_B)$ as a $B\otimes_{\bQ_p}K_0$-module. Thus, we have a surjection of $B\otimes_{\bQ_p}K_0$-modules
\[
f: \bigoplus_{i=1}^d (B\otimes_{\bQ_p}K_0) \cdot \hat{e}_i \twoheadrightarrow D_{\R{st}}^{K'}(V_B).
\]
As a $K_0$-vector space, $\R{dim}_{K_0}D_{\R{st}}^{K'}(V_B) = d\cdot \R{dim}_{\bQ_p}B$. Thus, $f$ is an isomorphism, and $D_{\R{st}}^{K'}(V_{B})$ is a free $B\otimes_{\bQ_p}K_0$-module of rank $d$.  
\end{proof}

Let $D_E$ be a finite $E$-vector space such that $D_{E, K} \coloneqq D_E\otimes_{\bQ_p}K$ is equipped with a decreasing filtration $\R{Fil}^iD_{E, K}$ of $E\otimes_{\bQ_p}K$-modules and $\{i ~|~ \R{gr}^iD_{E, K} \neq 0\} \subset \{0, \ldots, r\}$. We denote $\bm{v} = (D_{E, K}, \{\R{Fil}^i D_{E, K}\}_{i = 0, \ldots, r})$. We say that $V_B$ has \textit{Hodge-Tate type} $\bm{v}$ if $\R{gr}^i D_{\R{dR}}^{K}(V_B) \cong B\otimes_E \R{gr}^i D_{E, K}$ as $B_{K}$-modules for all $i$.\\

\noindent\textit{Remark.} We can consider a Hodge-Tate type as a conjugacy class of Hodge-Tate cocharacter in the following way. The Hodge-Tate period ring is given by $B_{\R{HT}} = C_p[X, X^{-1}]$ where $\cG_K$ acts on $X$ via the cyclotomic character $\chi$. When $V_B$ is de Rham and thus Hodge-Tate, we have the isomorphism $\alpha_{\R{HT}}: D_{\R{HT}}^K (V_B)\otimes_K B_{\R{HT}} \stackrel{\cong}{\longrightarrow} V_B \otimes_{\bQ_p} B_{\R{HT}}$, and $\displaystyle D_{\R{HT}}^K (V_B) \cong \bigoplus_i \R{gr}^i D_{\R{dR}}^{K}(V_B)$ as graded $B\otimes_{\bQ_p} K$-modules. By base change of $\alpha_{\R{HT}}$ via the map $B_{\R{HT}} \rightarrow C_p$ given by $X \mapsto 1$, we have the isomorphism $D_{\R{HT}}^K (V_B)\otimes_K C_p \cong V_B \otimes_{\bQ_p} C_p$. This gives the grading on $V_B \otimes_{\bQ_p} C_p$ whose graded pieces are $B\otimes_{\bQ_p} K$-modules. Therefore, we get the induced map $\mathbf{G}_m \rightarrow \R{Res}_{B/\bQ_p} (\R{GL}_B (V_B))$ over $C_p$, and we can think of a Hodge-Tate type as a conjugacy class of Hodge-Tate cocharacter. 

When $K/\bQ_p$ is finite and $E$ contains the Galois closure of $K$, then we can also consider a Hodge-Tate type as induced from embeddings of $K$ into $E$, but this point of view does not work in our case where $K/\bQ_p$ is allowed to be infinite. Note also that Hodge-Tate type does not make sense if we allow both $B$ and $K$ to be infinite over $\bQ_p$, so we always assume $B/\bQ_p$ is finite.\\

\begin{lem} \label{lem:3.2}
For a finite $B$-algebra $B'$, $V_{B'}$ has Hodge-Tate type $\bm{v}$ if $V_B$ has Hodge-Tate type $\bm{v}$.	
\end{lem}
 
\begin{proof}
It follows immediately from 	Lemma \ref{lem:3.1}.
\end{proof}
 
The goal of this subsection is to prove the following theorem:

\begin{thm} \emph{(cf. \cite[Theorem 4.3.4]{liu-filtration-torsion})} \label{thm:3.3}
There exists a constant $c_1$ depending only on $K', r$, and $d$ such that the following statement holds:

Let $A$ and $A'$ be finite flat $\cO_E$-algebras and let $\rho: \cG_K \arr \R{GL}_d(A)$ and $\rho': \cG_K \arr \R{GL}_d(A')$ be representations such that $\rho \in \R{Rep}_{A}^{\R{pst}, K', r}$ and $\rho' \in \R{Rep}_{A'}^{\R{pst}, K', r}$. Suppose that there exist an ideal $I \subset A$ such that $A/I$ is killed by a power of $p$ and an $\cO_E$-linear surjection $\beta: A' \twoheadrightarrow A/I$ such that $A/I \otimes_A \rho \cong A/I \otimes_{\beta, A'} \rho'$ as $A[\cG_K]$-modules. Let $V$ be the free $A[\frac{1}{p}]$-module of rank $d$ equipped with the $\cG_K$-action corresponding to $\rho\otimes_{\bZ_p}\bQ_p$, and similarly let $V'$ be corresponding to $\rho'\otimes_{\bZ_p}\bQ_p$. If $I \subset p^{c_1}A$ and $V'$ has Hodge-Tate type $\bm{v}$, then $V$ also has Hodge-Tate type $\bm{v}$. 	
\end{thm}

When $k$ is further assumed to be finite, Theorem \ref{thm:3.3} is proved in \cite[Theorem 4.3.4]{liu-filtration-torsion}. The proof in \cite{liu-filtration-torsion} is based on reducing to the case when $E$ contains the Galois closure of $K'$, and thus require $k$ to be finite. We remove such a restriction in the following.  

Since $E$ is a finite extension of $\bQ_p$, we have a ring isomorphism $E_{K} = E\otimes_{\bQ_p}K \cong \prod_{j=1}^s H_j$ for some fields $H_j$ finite over $K$. Note that each $H_j$ is an $E_{K}$-algebra via $E_{K} \cong \prod_{i=1}^s H_i \stackrel{q_j}{\arr} H_j$ where $q_j$ is the natural projection onto the $j$-th factor. For any $E_{K}$-module $M$, we write $M_j \coloneqq M\otimes_{E_{K}}H_j$. Then $M \cong \oplus_{j=1}^s M_j$. For a filtered $E_{K}$-module $D_{K}$, we denote $(\R{Fil}^iD_{K})_j$ and $(\R{gr}^i D_{K})_j$ by $\R{Fil}^i_j D_{K}$ and $\R{gr}^i_j D_{K}$ respectively. Since any finite $E_{K}$-module is projective, we have $\R{gr}^i_j D_{K} \cong \R{Fil}^i_j D_{K} / \R{Fil}^{i+1}_j D_{K}$. Write $B_{H_{j}} \coloneqq B \otimes_{E} H_j$.

\begin{lem} \emph{(cf. \cite[Lemma 4.1.4]{liu-filtration-torsion})} \label{lem:3.4}
With notations as above, $V_B$ has Hodge-Tate type $\bm{v}$ if and only if $\R{gr}^i_j D_{\R{dR}}^{K}(V_B)$ is $B_{H_j}$-free and $\R{rank}_{B_{H_j}} \R{gr}^i_j D_{\R{dR}}^{K}(V_B) = \R{dim}_{H_j}\R{gr}^i_j D_{E, K}$ for all $j = 1, \ldots, s$ and $i \in \bZ$.	
\end{lem}

\begin{proof}
This follows from the same argument as in the proof of  \cite[Lemma 4.1.4]{liu-filtration-torsion}.	
\end{proof}

$B_{\R{red}} = B/\cN(B)$ is a reduced Artinian $E$-algebra, so $B_{\R{red}} \cong \prod_{l=1}^m E_l$ for some field $E_l$ finite over $E$. We set $V_{E_l} = E_l\otimes_B V_B$.

\begin{lem} \emph{(cf. \cite[Proposition 4.1.5]{liu-filtration-torsion})} \label{lem:3.5}
$V_B$ has Hodge-Tate type $\bm{v}$ if and only if $V_{E_l}$ has Hodge-Tate type $\bm{v}$ for each $l= 1, \ldots, m$.	
\end{lem}

\begin{proof}
This follows from the same argument as in the proof of \cite[Proposition 4.1.5]{liu-filtration-torsion}.	
\end{proof}

The following lemma is useful when we consider an extension of the coefficient field $E$.

\begin{lem} \label{lem:3.6}
Let $H$ be a field, and let $C$ be a field (possibly of an infinite degree) over $H$. Let $H'$ be a finite extension of $H$, and let $R$ and $T$ be finite extensions of $H'$. If $M$ is a $C\otimes_H R$-module such that $M\otimes_{H'} T$ is a finite free $C\otimes_H R\otimes_{H'} T$-module, then $M$ is finite free over $C\otimes_H R$.	
\end{lem}

\begin{proof}
$M$ is a finite projective $C \otimes_H R$-module, and there exists a surjection $f: M \otimes_{H'} T \twoheadrightarrow M$ of $C \otimes_H R$-modules having a section. Let $\{e_1, \ldots, e_n\}$ be a basis of $M \otimes_{H'} T$ over $C\otimes_H R\otimes_{H'} T$. Let $N \coloneqq \oplus_{i=1}^n (C\otimes_H R) \cdot f(e_i)$. Then the natural map $N \arr M$ of $C \otimes_H R$-modules is an injection since $\{e_1, \ldots, e_n\}$ is a basis of $M \otimes_{H'} T$ over $C\otimes_H R\otimes_{H'} T$. Furthermore, $\R{dim}_C N = \R{dim}_C M$, so it is bijective. 

\end{proof}

Let $L$ be a finite extension of $E$, and write $B_L \coloneqq L\otimes_E B$. Given $\bm{v}$ as above, let $\bm{v}' = (D_L \coloneqq L\otimes_E D, ~\{\R{Fil}^i D_{L, K} = L\otimes_E \R{Fil}^i D_{E, K}\}_{i=0, \ldots, r})$.

\begin{lem} \emph{(cf. \cite[Lemma 4.1.6]{liu-filtration-torsion})} \label{lem:3.7}
With notations as above, $V_B$ has Hodge-Tate type $\bm{v}$ if and only if $V_{B_L} \coloneqq B_L\otimes_B V_B$ has Hodge-Tate type $\bm{v}'$.
\end{lem}

\begin{proof}
Given Lemma \ref{lem:3.6}, it follows from the same argument as in the proof of \cite[Lemma 4.1.6]{liu-filtration-torsion}.
\end{proof}

\begin{lem} \label{lem:3.8}
Suppose we have an injection $B \hookrightarrow B'$ of finite $E$-algebras. If $V_{B'} = B'\otimes_B V_B$ has Hodge-Tate type $\bm{v}$, then also $V_B$ has Hodge-Tate type $\bm{v}$.	
\end{lem}

\begin{proof}
We have an induced injection of finite $E$-algebras $B_{\R{red}} \hookrightarrow B'_{\R{red}}$. By Lemma \ref{lem:3.5}, we can reduce to the case when $B$ and $B'$ are fields. Then it follows from Lemma \ref{lem:3.4} and Lemma \ref{lem:3.6}.	
\end{proof}

As we will apply the functor $M_{\R{st}}$ to $\cG_K$-representations semi-stable over $K'$, we need to consider $D_{\R{dR}}^{K'}(V_B) \coloneqq (B_{\R{dR}}\otimes_{\bQ_p} V_B^{\vee})^{\cG_{K'}}$. Note that $D_{\R{dR}}^{K'}(V_B) = D_{\R{dR}}^K (V_B) \otimes_K K'$. Thus, by essentially the same argument as in the proof of Lemma \ref{lem:3.7}, we see that $V_B$ has Hodge-Tate type $\bm{v}$ if and only if $\R{gr}^i D_{\R{dR}}^{K'}(V_B) \cong B\otimes_E \R{gr}^i D_{E, K'}$ as $B_{K'}$-modules for all $i$. Here, $D_{E, K'} \coloneqq D_E \otimes_{\bQ_p} K' = D_{E, K} \otimes_K K'$ which has the induced filtration from $D_{E, K}$. 

Let $K_1 \subset E$ be the maximal unramified subextension over $\bQ_p$. Then $K_1 = W(k_1)[\frac{1}{p}]$ for some finite field $k_1$, and $E/K_1$ is totally ramified. Choose a uniformizer $\varpi_E \in E$, and let $\tilde{F}(u)$ be its Eisenstein polynomial over $K_1$. Let $G(u)$ be a monic irreducible polynomial in $\bQ_p[u]$ such that $K_1 \cong \bQ_p[u]/G(u)\bQ_p[u]$, and let $G(u) = \prod_{j=1}^m G_j(u)$ be the decomposition into monic irreducible polynomials in $K_0[u]$. Note that $G_j(u) \in W(k)[u]$. Denote by $\bar{G}_j(u) \in k[u]$ the reduction of $G_j(u)$ mod $p$. Then $\bar{G}_j(u)$ is irreducible in $k[u]$ and $k[u]/\bar{G}_j(u)k[u] \cong l_j$ for a finite extension $l_j / k$. By the Chinese remainder theorem, $W(k_1)\otimes_{\bZ_p}W(k) \cong \prod_{j=1}^m W(l_j)$. Since $\tilde{F}(u)$ is irreducible over $W(l_j)[\frac{1}{p}]$ for each $j$, we have $E\otimes_{\bQ_p}K_0 \cong \prod_{j=1}^m L_j$ and $\cO_E\otimes_{\bZ_p}W(k) \cong \prod_{j=1}^m \cO_{L_j}$ where $L_j \coloneqq (W(l_j)[\frac{1}{p}])(\varpi_E)$.

For each $j = 1, \ldots, m$, let $F(u) = \prod_{s=1}^{t_j} F_{j_s}(u)$ be the decomposition of $F(u)$ into monic irreducible polynomials in $L_j[u]$, and choose a root $\varpi_{j_s}$ of $F_{j_s}(u)$ for each $s$. Then  
\[
L_j\otimes_{K_0} K' \cong \prod_{s=1}^{t_j} L_j[u]/F_{j_s}(u)L_j[u] \cong \prod_{s=1}^{t_j} T_{j_s}
\]
where $T_{j_s} \coloneqq L_j(\varpi_{j_s})$. Thus, we have ring isomorphisms
\[
E\otimes_{\bQ_p}K' \cong (E\otimes_{\bQ_p}K_0)\otimes_{K_0}K' \cong \prod_{j=1}^m\prod_{s=1}^{t_j}T_{j_s}.
\]
Let $t = \sum_{j=1}^m t_j$. After re-indexing the fields $T_{j_s}$, we have $E\otimes_{\bQ_p}K' \cong \prod_{s=1}^t T_j$, and the statement analogous to Lemma \ref{lem:3.4} holds for $D_{\R{dR}}^{K'}(V_B)$. 

Let $\cO_{E, K'} \coloneqq \cO_E\otimes_{\bZ_p}\cO_{K'}$. The projection $q_s: E_{K'} \arr T_s$ induces the map $\cO_{E, K'} \arr \cO_{T_s}$, and we have the natural map $q: \cO_{E, K'} \arr \prod_{s=1}^t \cO_{T_s}$. Denote by $v_p$ the $p$-adic valuation normalized by $v_p(p) = 1$.

\begin{lem} \label{lem:3.9}
There exists a positive integer $c'$ depending only on $K_0$ and $F(u)$ such that $p^{c'}(\prod_{s=1}^t \cO_{T_s}) \subset q(\cO_{E, K'})$.	
\end{lem}

\begin{proof}
For a field $L$ finite over $K_0$, let $F(u) = \prod_{s=1}^w F_s(u)$ be the decomposition of $F(u)$ into monic irreducible polynomials in $L[u]$, and choose a root $\varpi_s$ of $F_s(u)$ for each $s$. We have 
\[
L\otimes_{K_0} K' \cong \prod_{s=1}^w L[u]/F_s(u)L[u] \cong \prod_{s=1}^w L_s'
\]
where $L_s' \coloneqq L(\varpi_s)$. Let $q_s': L\otimes_{K_0} K' \arr L_s'$ be the composition of the above isomorphism followed by the projection onto the $s$-th factor. Then $q_s'$ induces a surjection $\cO_{L}\otimes_{\bZ_p} \cO_{K'} \twoheadrightarrow \cO_{L_s'}$, and we have the natural map $\cO_{L}\otimes_{\bZ_p} \cO_{K'} \hookrightarrow \prod_{s=1}^w \cO_{L_s'}$. Under this map, $\prod_{h \neq s} F_h(\pi)$ maps to $(0, \ldots, 0, \prod_{h \neq s}F_h(\varpi_s), 0, \ldots, 0)$ whose components are $0$ except the $s$-th component. Write $v_p(\prod_{h \neq s}F_h(\varpi_s)) = \frac{a}{b}$ for some relatively prime positive integers $a, b$. Then $(\prod_{h \neq s}F_h(\varpi_s))^b = p^a x$ for some $x \in \cO_{L_s}^{\times}$ with $v_p(x) = 0$. Thus, $(0, \ldots, 0, p^a, 0, \ldots, 0)$, whose components are $0$ except the $s$-th component, lies in the image of $\cO_{L}\otimes_{\bZ_p}\cO_{K'}$ under the above map. 

Repeating this argument for all $s = 1, \ldots, w$ and considering all possible decompositions of $F(u)$ into irreducible factors over some finite field over $K_0$, we see that there exists a positive integer $c'$ depending only on $K_0$ and $F(u)$ such that for any $L$ finite over $K_0$, if we write $L\otimes_{K_0}K \cong \prod_{s=1}^w L(\varpi_s)$ as above, then for each $s$, $(0, \ldots, 0, p^{c'}, 0, \ldots, 0)$ whose components are $0$ except the $s$-th component lies in the image of $\cO_{L}\otimes_{\bZ_p}\cO_{K'}$. Applying this for each $L_j$, we get the result.
\end{proof}

\begin{cor} \label{cor:3.10}
Let $M$ be a torsion free $\cO_{E, K'}$-module. Then the torsion part of $M_s \coloneqq M \otimes_{\cO_{E, K'}, q_s} \cO_{T_s}$ is killed by $p^{c'}$, where $c'$ is the constant given in Lemma \ref{lem:3.9}. 	
\end{cor}

\begin{proof}
Let $M' = \oplus_{s=1}^t M_s$. By Lemma \ref{lem:3.9}, there exist morphisms of $\cO_{E, K'}$-modules $q_M: M \arr M'$ and $s_M: M' \arr M$ such that $q_M \circ s_M = p^{c'}\R{Id}|_{M'}$. Let $x$ be a torsion element in $M'$. Then $s_M(x) = 0$, so $p^{c'}x = q_M(s_M(x)) = 0$.	
\end{proof}

Let $C$ be a finite flat $\cO_E$-algebra, and let $\Lambda \in \R{Rep}_C^{\R{pst}, K', r}$ such that $\Lambda$ is a finite free $C$-module of rank $d$ and $\Lambda[\frac{1}{p}]$ has Hodge-Tate type $\bm{v}$. Since $\cO_E$ is henselian, $C \cong \prod_{j=1}^n C_j$ where each $C_j$ is a finite flat local $\cO_E$-algebra. We say $C$ is \textit{good} if for each $j = 1, \ldots, n$, there exists a prime ideal $\mathfrak{p}_j \subset C_j$ such that $C_j/\mathfrak{p}_j \cong \cO_{F_j}$ for some finite extension $F_j/\bQ_p$. Let $L_{K'} \coloneqq M_{\R{st}}(\Lambda)_{K'}$.

\begin{lem} \label{lem:3.11}
Suppose $C$ is good. Then $L_{K'}$ is finite free over $C_{K'} \coloneqq C\otimes_{\bZ_p} \cO_{K'}$ of rank $d$. 
\end{lem}

\begin{proof}
By Theorem \ref{thm:2.5}, there exists a unique Kisin module $\fM \in \R{Mod}_{\fS}^r(\p, \hat{\cG})$ such that $\hat{T}^{\vee}(\hat{\fM}) = \Lambda$. Write $\fS_C \coloneqq C\otimes_{\bZ_p}\fS$. By the construction of the functor $M_{\R{st}}$ in Section \ref{sec:2.2}, it suffices to show that $\fM$ is a finite free $\fS_C$-module of rank $d$. The Kisin module corresponding to $C_j \otimes_C \Lambda$ is $C_j\otimes_C \fM$, so we may assume without loss of generality that $n = 1$ and so that $C$ is a local ring. 
 	
The Kisin module corresponding to $\cO_{F_1}\otimes_C \Lambda$ (via $C/\mathfrak{p}_1 \cong \cO_{F_1}$) is $\fM' \coloneqq \cO_{F_1} \otimes_C \fM$. Since $\fM'$ is finite free over $\fS$, $\fM'/u\fM'$ is $p$-torsion free. Thus, $\fM'/u\fM'$ is a projective $\cO_{F_1}\otimes_{\bZ_p}W(k)$-module. Since $(\fM'/u\fM')[\frac{1}{p}]$ is isomorphic to its pullback by $\p$ and $\p$ permutes the maximal ideals of $\cO_{F_1}\otimes_{\bZ_p}W(k)$ transitively, $\fM'/u\fM'$ is a free $\cO_{F_1}\otimes_{\bZ_p}W(k)$-module of rank $d$. Thus, $\fM'$ is a free $\cO_{F_1}\otimes_{\bZ_p}\fS$-module of rank $d$.

By Nakayama's lemma, we have a surjection
\[
f: \bigoplus_{i=1}^d \fS_C \cdot e_i \twoheadrightarrow \fM
\]
of $\fS_C$-modules. $\Lambda$ is a free $\bZ_p$-module of rank $[C : \bZ_p]d$, so $\fM$ is free over $\fS$ of rank $[C : \bZ_p]d$. Thus, $f$ is an isomorphism.
\end{proof}

Suppose that $C$ is good and that there exists an ideal $J \subset C$ such that $C/J \cong \cO_E/p^b\cO_E$ for some positive integer $b$. For $s = 1, \ldots, t$, we set $C[\frac{1}{p}]_s \coloneqq (C[\frac{1}{p}]\otimes_{\bQ_p}K')\otimes_{E_{K'}, q_s}T_s$, and define $d_s \coloneqq \R{rank}_{C[\frac{1}{p}]_s}(\R{gr}^0_s (D_{\R{dR}}^{K'}(\Lambda[\frac{1}{p}]))).$ Denote $\R{Fil}^i_s L_{K'} \coloneqq \R{Fil}^i L_{K'}\otimes_{\cO_{E, K'}, q_s} \cO_{T_s}$, and similarly for the graded modules. By Lemma \ref{lem:3.11}, $\R{Fil}^0_s L_{K'}$ is free over $C_s \coloneqq C_{K'}\otimes_{\cO_{E, K'}, q_s} \cO_{T_s}$ of rank $d$.

\begin{lem} \emph{(cf. \cite[Lemma 4.2.7]{liu-filtration-torsion})} \label{lem:3.12}
Suppose that $d_s \neq 0$. Let $l$ be a positive integer satisfying $b \geq ld+1$. Then there exists $x \in \R{gr}^0_s L_{K'}/J\R{gr}^0_s L_{K'}$ such that $p^lx \neq 0$.	
\end{lem}

\begin{proof}
This follows from essentially the same argument as in the proof of \cite[Lemma 4.2.7]{liu-filtration-torsion}. For any $C$-module $M$, denote $M/JM$ by $M/J$. We have the following right exact sequence:
\[
 \R{Fil}^1_s L_{K'} \arr \R{Fil}^0_s L_{K'} \arr \R{gr}^0_s L_{K'} \arr 0.
\]	
Let $\tilde{\R{Fil}}^1_s L_{K'}$ be the image of $\R{Fil}^1_s L_{K'}$ in $\R{Fil}^0_s L_{K'}$ under the first map in the above sequence. We then obtain the following right exact sequence
\[
\tilde{\R{Fil}}^1_s L_{K'}/J \arr \R{Fil}^0_s L_{K'}/J \arr \R{gr}^0_s L_{K'}/J \arr 0.
\]
Denote $\bar{M} \coloneqq \R{Fil}^0_s L_{K'}/J$ and let $\bar{N} \subset \bar{M}$ be the submodule given by the image of $\tilde{\R{Fil}}^1_s L_{K'}/J$. Then $\bar{M}/\bar{N} = \R{gr}^0_s L_{K'}/J$. 

Suppose that $p^l$ annihilates $\bar{M}/\bar{N}$. By Lemma \ref{lem:3.11}, $\bar{M}$ is a finite free $\cO_{T_s}/p^b\cO_{T_s}$-module of rank $d$. Let $\bar{\pi}_s$ be a uniformizer of $\cO_{T_s}$. Then there exists an $\cO_{T_s}/p^b\cO_{T_s}$-basis $\bar{e}_1, \ldots, \bar{e}_d$ of $\bar{M}$ such that
\[
\bar{N} \cong \bigoplus_{i=1}^d (\cO_{T_s}/p^b\cO_{T_s})\cdot(\bar{\pi}_s^{a_i}\bar{e}_i) 
\]
for some nonnegative integers $a_i$. We have $\bar{\pi}_s^{a_i} ~|~ p^l$ for all $i = 1, \ldots, d$. Let $e_1, \ldots, e_d$ be a $C_s$-basis of $\R{Fil}^0_s L_{K'}$ which lifts $\bar{e}_1, \ldots, \bar{e}_d$. For $i = 1, \ldots, d$, let $y_i \in \tilde{\R{Fil}}^1_s L_{K'}$ which lifts $\bar{\pi}_s^{a_i}\bar{e}_i$. If $X$ denotes the $d \times d$-matrix such that $(y_1, \ldots, y_d) = (e_1, \ldots, e_d)X$, then $\R{det}(X) = \bar{\pi}_s^a+j$ with $a = \sum_{i=1}^d a_i$ and $j \in J$. Since $b \geq ld+1$, we have $\bar{\pi}_s^a \neq 0$ in $C_s/J$, and thus $\R{det}(X) \neq 0$ in $C_s$. On the other hand, let $\bar{z}_1, \ldots, \bar{z}_{d_s}$ be a $C[\frac{1}{p}]_s$-basis of $\R{gr}^0_s(D_{\R{dR}}^{K'}(\Lambda[\frac{1}{p}]))$. We have $\R{det}(X)(e_1, \ldots, e_d) \subset \R{Fil}^1_s (D_{\R{dR}}^{K'}(\Lambda[\frac{1}{p}]))$, and therefore $\R{det}(X)\bar{z}_i = 0$. This gives a contradiction.   

\end{proof}

\noindent\textit{Proof of Theorem \ref{thm:3.3}}. Given above results, Theorem \ref{thm:3.3} follows from essentially the same argument as in the proof of \cite[Theorem 4.3.4]{liu-filtration-torsion}, except that we do not reduce to the case where $E$ contains the Galois closure of $K'$. We also remark that the proof of \cite{liu-filtration-torsion} reduces to the case $A'$ is local. But $A'$ is not necessarily finite over $\cO_E$ after such reduction, which has been overlooked in \cite{liu-filtration-torsion}. This is a very minor gap, and we remedy it by only reducing to the case $A'$ is good.

We first reduce to the case where $A = \cO_E$ and $A'$ is good. For this, let $B \coloneqq A \otimes_{\bZ_p}\bQ_p$ and $B' \coloneqq A'\otimes_{\bZ_p}\bQ_p$. We have $B_{\R{red}} = B/\cN(B) \cong \prod_{j} E_j$ and $B'_{\R{red}} \cong \prod E_j'$ for some $E_j, E_j'$ finite over $E$. Let $L$ be a finite Galois extension of $E$ containing all Galois closures of $E_j, E_j'$. Denote $\cO_L\otimes_{\cO_E}(*)$ by $(*)_{\cO_L}$ for $(*)$ being $A, A', \rho, \rho', I$, and $\beta$. Note that $(A_{\cO_L}[\frac{1}{p}])_{\R{red}} = L\otimes_E B_{\R{red}} = L\otimes_E \prod E_j \cong \prod_i L$ with $E_j$ embedding into $L$ differently, and similarly for $(A'_{\cO_L}[\frac{1}{p}])_{\R{red}}$. This induces the natural map $\psi_l: A_{\cO_L} \arr (A_{\cO_L})[\frac{1}{p}] \twoheadrightarrow L$ to the $l$-th factor of $\prod_i L$. By Lemma \ref{lem:3.5} and \ref{lem:3.7}, it suffices to show (assuming $I \subset p^{c_1}A$ for a suitable constant $c_1$) that $L\otimes_{\psi_l, A_{\cO_L}}(\rho)_{\cO_L}$ has Hodge-Tate type $\bm{v}$. Let $A_l = \psi_l(A_{\cO_L})$ and $I_l = \psi_l(I_{\cO_L})$. $\psi_l: A_{\cO_L} \twoheadrightarrow A_l \subset L$ is a morphism of $\cO_L$-algebras, so $A_l = \cO_L$ (and analogously for $A'_{\cO_L}$), and we have a natural projection $\gamma_l: A_{\cO_L}/I_{\cO_L} \twoheadrightarrow A_l/I_l$. Thus, by replacing $E$ by $L$ and $A'$ by $A'_{\cO_L}$, we can assume that $A = \cO_E$ and that $A'$ is good.

Let $T$ denote the torsion representation $A/I\otimes_A \rho \cong A'/I'\otimes_{A'} \rho' \in \R{Rep}_{\R{tor}, \cO_E}^{\R{pst}, K', r}$ where $I' = \R{ker}(\beta)$. We denote by $j$ and $j'$ the two lifts $\rho$ and $\rho'$ of $T$ respectively. Write $L_{K'} \coloneqq M_{\R{st}}(\rho)_{K'}, L'_{K'} \coloneqq M_{\R{st}}(\rho')_{K'}, M_{K'} \coloneqq M_{\R{st}, j}(T)_{K'}$, and $M'_{K'} \coloneqq M_{\R{st}, j'}(T)_{K'}$. We have $\R{gr}^i_sM_{K'} \cong \R{gr}^i_s L_{K'} / I\R{gr}^i_sL_{K'}$ and $\R{gr}^i_sM'_{K'} \cong \R{gr}^i_s L'_{K'} / I'\R{gr}^i_sL'_{K'}$ for $s = 1, \ldots, t$. By Corollary \ref{cor:2.8} and Proposition \ref{prop:2.9}, there exist morphisms of $\cO_{T_s}$-modules $g_s^i: \R{gr}^i_sM_{K'} \arr \R{gr}^i_sM'_{K'}$ and $h_s^i: \R{gr}^i_sM'_{K'} \arr \R{gr}^i_sM_{K'}$ such that $g_s^i \circ h_s^i = p^{2c} \R{Id}|_{\R{gr}^i_sM'_{K'}}$ and $h_s^i \circ g_s^i = p^{2c}\R{Id}|_{\R{gr}^i_sM_{K'}}$.

Now, we set $\tilde{c} = \tilde{c}(K', r, d) \coloneqq (2c+c')d+1$ where $c$ and $c'$ are given as in Theorem \ref{thm:2.7} and Lemma \ref{lem:3.9} respectively. Assume $I \subset p^{\tilde{c}}A = p^{\tilde{c}}\cO_E$. We claim that if $\R{gr}^0_s(D_{\R{dR}}^{K'}(V')) \neq 0$, then $\R{gr}^0_s(D_{\R{dR}}^{K'}(V)) \neq 0$. Suppose $\R{gr}^0_s(D_{\R{dR}}^{K'}(V)) = 0$. By Corollary \ref{cor:3.10}, $\R{gr}^0_s M_{K'}$ is killed by $p^{c'}$. But by Lemma \ref{lem:3.12}, there exists $x \in \R{gr}^0_s M'_{K'}$ such that $p^{c'+2c}x \neq 0$. We have a contradiction since $p^{c'+2c}x = g^0_s(p^{c'}h^0_s(x))$. 

On the other hand, let $B' = A'[\frac{1}{p}]$, and denote $d_0 = \R{dim}_{T_s} \R{gr}^0_s(D_{\R{dR}}^{K'}(V))$. We claim (assuming $I \subset p^{\tilde{c}}\cO_E$) that $d_0 \leq \R{dim}_{T_s} \R{gr}^0_s(D_{\R{dR}}^{K'}(V'))$. For this, note that as an $\cO_{T_s}$-module, $\R{gr}_s^0L_{K'} = N_{\R{tor}}\oplus N$ where $N_{\R{tor}}$ is the torsion submodule of $\R{gr}_s^0L_{K'}$ and $N$ is a finite free $\cO_{T_s}$-module of rank $d_0$. By Corollary \ref{cor:3.10}, 
\[
\R{gr}^0_s M_{K'} \cong N_{\R{tor}}\oplus \bigoplus_{i=1}^{d_0} \cO_{T_s}/I\cO_{T_s}.
\]
Let $\bar{N} \coloneqq p^{c'}\bigoplus_{i=1}^{d_0} \cO_{T_s}/I\cO_{T_s}$. Then $p^{c'}\R{gr}^0_s M_{K'} = \bar{N}$, again by Corollary \ref{cor:3.10}, and therefore $h_s^0(g_s^0(\bar{N})) \cong \bigoplus_{i=1}^{d_0} p^{2c+c'}\cO_{T_s}/I\cO_{T_s}$. Since $p^{c'}\R{gr}_s^0L'_{K'}$ surjects onto $h_s^0(p^{c'}\R{gr}_s^0M'_{K'})$ and $g_s^0(\bar{N}) \subset p^{c'}\R{gr}_s^0M'_{K'}$, we have by Corollary \ref{cor:3.10} that the $\cO_{T_s}$-rank of $p^{c'}\R{gr}_s^0L'_{K'}$ is at least $d_0$. Thus, the $\cO_{T_s}$-rank of $\R{gr}_s^0L'_{K'}$ is at least $d_0$, and $\R{dim}_{T_s} \R{gr}^0_s(D_{\R{dR}}^{K'}(V')) \geq d_0$.    

Hence, assuming $I \subset p^{\tilde{c}}\cO_E$, we have $\R{gr}^0_s(D_{\R{dR}}^{K'}(V)) \neq 0$ if and only if $\R{gr}^0_s(D_{\R{dR}}^{K'}(V')) \neq 0$.  

For the last step, we set $c_1 = \tilde{c}(K', dr, d)$ and assume $I \subset p^{c_1}\cO_E$. It suffices to show that for each $i$,
\[
\R{dim}_{T_s} \R{gr}^i_s(D_{\R{dR}}^{K'}(V)) = \R{rank}_{B'_{T_s}} \R{gr}^i_s(D_{\R{dR}}^{K'}(V')).
\]
Suppose that the above equality fails for some $i$, and let $i_*$ be the smallest such number. Write $d_i = \R{dim}_{T_s} \R{gr}^i_s(D_{\R{dR}}^{K'}(V))$ and $d_i' = \R{rank}_{B'_{T_s}} \R{gr}^i_s(D_{\R{dR}}^{K'}(V'))$. Suppose first $d_{i_*} > d'_{i_*}$. We set $t_1 = \sum_{i\leq i_*}d_i$ and $t_2 = \sum_{i\leq i_*}id_i$. Let $\tilde{i} = \R{max}\{i ~|~ \sum_{j\leq i}d'_j \leq t_1\}$ and $t' = \sum_{i \leq \tilde{i}}d'_i$. Then $i_* \leq \tilde{i}$ and $t' \leq t_1$. Let 
\[
t'' = (\sum_{i\leq\tilde{i}}id'_i)+(t_1-t')(\tilde{i}+1).
\]
We have $t_2 < t''$. Moreover, $t_2$ (resp. $t''$) is the smallest $i$ such that $\R{gr}^i_s(D_{\R{dR}}^{K'}(\bigwedge^{t_1}V))$ (resp. $\R{gr}^i_s(D_{\R{dR}}^{K'}(\bigwedge^{t_1}V'))$) is nontrivial. Let $\chi$ be a crystalline character such that $\R{gr}^i_s(D_{\R{dR}}^{K'}(\chi)) \neq 0$ only when $i = -t_2$. Then $\R{gr}^0_s(D_{\R{dR}}^{K'}(\chi\bigwedge^{t_1}V))$ is nontrivial. From the above result applied to $\chi\bigwedge^{t_1}V$ and $\chi\bigwedge^{t_1}V'$, we see that $\R{gr}^0_s(D_{\R{dR}}^{K'}(\chi\bigwedge^{t_1}V'))$ is also nontrivial, leading to a contradiction. 

By switching the roles of $V$ and $V'$, it follows similarly that we cannot have $d_{i_*} < d'_{i_*}$. This completes the proof. $\hspace{26em}\qed$

\subsection{Galois Type} \label{sec:3.2}

We now study the Galois types of potentially semi-stable representations. As in Section \ref{sec:3.1}, let $E$ be a finite field over $\bQ_p$, and let $B$ be a finite $E$-algebra. Let $V_B$ be a free $B$-module of rank $d$ equipped with a potentially semi-stable continuous $\cG_K$-action. Let 
\[
D_{\R{pst}}(V_B) = \lim_{\stackrel{\longrightarrow}{K \subset K''}} (B_{\R{st}}\otimes_{\bQ_p}V_B^{\vee})^{\cG_{K''}},
\]
where the limit goes over finite extensions $K''$ of $K$ contained in $\bar{K}$. Denote by $K_0^{\R{ur}}$ the union of finite unramified extensions of $K_0$ contained in $\bar{K}$. We have $\R{dim}_{K_0^{\R{ur}}} D_{\R{pst}}(V_B) = \R{dim}_{\bQ_p} V_B$. 

\begin{lem} \label{lem:3.13}
Let $B'$ be a finite $B$-algebra, and write $V_{B'} = B'\otimes_B V_B$. Then $V_{B'}$ is potentially semi-stable as a $\cG_K$-representation, and $D_{\R{pst}}(V_{B'}) \cong B'\otimes_B D_{\R{pst}}(V_B)$. If $V_B$ becomes semi-stable over $L \supset K$, then so does $V_{B'}$. Furthermore, $D_{\R{pst}}(V_B)$ is a free $B\otimes_{\bQ_p}K_0^{\R{ur}}$-module of rank $d$. 	
\end{lem}

\begin{proof}
It follows from essentially the same proof as for Lemma \ref{lem:3.1}.	
\end{proof}

$D_{\R{pst}}(V_B)$ is equipped with a $K_0^{\R{ur}}$-semilinear action of $\cG_K$, and thus a $K_0^{\R{ur}}$-linear action of the inertia group $I_K$. The Frobenius action commutes with the $I_K$-action, so $\R{tr}(\sigma|D_{\R{pst}}(V_B)) \in B$ for all $\sigma \in I_K$. 

Let $D_E$ be an $E$-vector space of dimension $d$, and let $D_{E, K} = D_E \otimes_{\bQ_p} K$ equipped with a filtration giving a Hodge-Tate type $\bm{v}$. Fix a representation 
\[
\tau: I_K \arr \R{End}_E (D_E)
\]
with an open kernel. Note that there exists an $I_K$-stable $\cO_E$-lattice in $D_E$, so $\R{tr}(\tau(\sigma))\in \cO_E$ for all $\sigma \in I_K$. We say $V_B$ has \textit{Galois type} $\tau$ if the $I_K$-representation $D_{\R{pst}}(V_B)$ is equivalent to $\tau$, i.e., $\R{tr}(\sigma | D_{\R{pst}}(V_B)) = \R{tr}(\tau(\sigma))$ for all $\sigma \in I_K$.

Let $L/K$ be a finite Galois extension contained in $\bar{K}$ such that $I_L \subset \R{ker}(\tau)$. Here, $I_L$ denotes the inertia subgroup of $\cG_L$. $D_{\R{st}}^L(V_B) = (B_{\R{st}}\otimes_{\bQ_p}V_B^{\vee})^{\cG_L}$ is an $L_0$-vector space where $L_0$ is the maximal unramified subextension of $K_0$ contained in $L$. If $V_B$ is semi-stable over $L$, then $D_{\R{pst}}(V_B) \cong K_0^{\R{ur}}\otimes_{L_0}D_{\R{st}}^L(V_B)$. Therefore, $V_B$ has Galois type $\tau$ if and only if $V_B$ is semi-stable over $L$ and $\R{tr}(\sigma|D_{\R{st}}^L(V_B)) = \R{tr}(\tau(\sigma))$ for all $\sigma \in I_{L/K}$, where $I_{L/K}$ is the inertia subgroup of $\R{Gal}(L/K)$.

\begin{lem} \label{lem:3.14}
Let $\alpha: B \arr B'$ be an $E$-algebra morphism between finite $E$-algebras. Suppose $V$ is semi-stable over $L$. Then for all $\sigma \in I_{L/K}$, we have $\R{tr}(\sigma|D_{\R{st}}^L(V_{B'})) = \alpha(\R{tr}(\sigma|D_{\R{st}}^L(V_{B})))$. In particular, if $V_B$ has Galois type $\tau$, then so does $V_{B'}$. If $\alpha$ is injective, then the converse is also true, i.e., $V_B$ has Galois type $\tau$ if and only if $V_{B'}$ has Galois type $\tau$. 	
\end{lem}

\begin{proof}
$D_{\R{pst}}(V_{B'}) \cong B'\otimes_B D_{\R{pst}}(V_B)$ by Lemma \ref{lem:3.13}, so 
\[
\R{tr}(\sigma|D_{\R{st}}^L(V_{B'})) = \alpha(\R{tr}(\sigma|D_{\R{st}}^L(V_{B})))
\]
for all $\sigma \in I_{L/K}$. The remaining statements follow immediately.   	
\end{proof}

Consider the case when $B$ is local. If $E'$ is its residue field, then $E'$ is finite over $E$ and $B$ is naturally an $E'$-algebra. Note that the $I_K$-action on $D_{\R{pst}}(V_B)$ has an open kernel. Since the cohomology of a finite group with coefficients in $E'\otimes_{\bQ_p} K_0^{\R{ur}}$ is trivial in all positive degrees, it follows from the deformation theory that the representation $D_{\R{pst}}(V_B)$ arises from a representation over $E'\otimes_{\bQ_p}K_0^{\R{ur}}$. Thus, $V_B$ has Galois type $\tau$ if and only if $V_{E'} = E'\otimes_B V_B$ has Galois type $\tau$. For a general finite $E$-algebra $B$, we have isomorphisms $B \cong \prod_{i=1}^n B_{\fm_i}$ and $B_{\R{red}} \cong \prod_{i=1}^n E_i$, where $\fm_1, \ldots, \fm_n$ are the maximal ideals of $B$ and $E_i = B_{\fm_i}/\fm_iB_{\fm_i}$. Let $V_{E_i} = E_i \otimes_B V_B$. The following lemmas are analogous to Lemma \ref{lem:3.5} and \ref{lem:3.7}.

\begin{lem} \label{lem:3.15}
$V_B$ has Galois type $\tau$ if and only if $V_{E_i}$ has Galois type $\tau$ for each $i = 1, \ldots, n$.	
\end{lem}

\begin{proof}
It follows directly from Lemma \ref{lem:3.14}.
\end{proof}

\begin{lem} \label{lem:3.16}
Let $E'$ be a finite extension of $E$, and let $B_{E'} = E'\otimes_E B$ and $V_{B_{E'}} = B_{E'}\otimes_B V_B$. Then $V_B$ has Galois type $\tau$ if and only if $V_{B_{E'}}$ has Galois type $\tau$.	
\end{lem}

\begin{proof}
Since the natural map of $E$-algebras $B \arr B_{E'}$ is injective, it follows from Lemma \ref{lem:3.14}.	
\end{proof}

The following theorem is essential in studying the locus of representations with a given Galois type. 

\begin{thm} \label{thm:3.17}
Let $\tau$ be a Galois type, and let $L/K$ be a finite Galois extension in $\bar{K}$ over which $\tau$ becomes trivial. Let $A$ be a finite flat $\cO_E$-algebra and $\rho: \cG_K \arr \emph{GL}_d(A)$ be a Galois representation such that $\rho \otimes_{\bZ_p} \bQ_p$ is semi-stable over $L$ having Hodge-Tate weights in $[0, r]$.

Suppose that for each positive integer $n$, there exist a finite flat $\cO_E$-algebra $A_n$, a Galois representation $\rho_n: \cG_K \arr \emph{GL}_d(A_n)$, and an $\cO_E$-linear surjection $\beta_n: A_n \arr A/p^n A$ such that $A/p^n A \otimes_A \rho \cong A/p^n A \otimes_{\beta_n, A_n} \rho_n$ as $A[\cG_K]$-modules, and that $\rho_n \otimes_{\bZ_p} \bQ_p$ is semi-stable over $L$ having Hodge-Tate weights in $[0, r]$ and Galois type $\tau$. 

Then $\rho \otimes_{\bZ_p} \bQ_p$ also has Galois type $\tau$. 	
\end{thm}

\begin{proof}
Let $B = A[\frac{1}{p}]$. We have $B_{\R{red}} \cong \prod_i E_i$ for some finite extensions $E_i/E$. Let $H / E$ be a finite Galois extension containing the Galois closures of $E_i$ for all $i$. We write $A_{\cO_H} = \cO_H \otimes_{\cO_E} A$. Then $(A_{\cO_H}[\frac{1}{p}])_{\R{red}} \cong H\otimes_E B_{\R{red}} \cong H \otimes_E \prod_i E_i$. Since $H$ contains the Galois closures of $E_i$ for all $i$, $H\otimes_E E_i \cong \prod_j H$ with $E_i$ embedding to $H$ differently. This induces the natural map $\psi_l: A_{\cO_H} \arr A_{\cO_H}[\frac{1}{p}] \arr H$ to the $l$-th factor of $\prod_j H$. Let $A_l = \psi_l(A_{\cO_H})$. Since $\psi_l: A_{\cO_H} \twoheadrightarrow A_l \subset H$ is a morphism of $\cO_H$-algebras, $A_l = \cO_H$. By Lemma \ref{lem:3.15} and \ref{lem:3.16}, it suffices to show that $H\otimes_{\psi_l, A_{\cO_H}} (A_{\cO_H} \otimes_A \rho)$ has Galois type $\tau$. Therefore, we may and will replace $A$ by $\cO_H$, $\rho$ by $\cO_H \otimes_{\psi_l, A_{\cO_H}} (A_{\cO_H} \otimes_A \rho)$, $A_n$ by $\cO_H\otimes_{\cO_E} A_n$, and replace $\beta_n$ and $\rho_n$ accordingly.

Denote by $L_0$ the maximal unramified extension of $K_0$ contained in $L$. Note that $I_{L/K} \cong I_{L/KL_0}$. Applying the results of Section \ref{sec:2} with $(L_0, KL_0, L)$ in place of $(K_0, K, K')$, we get the associated lattice $M_{\R{st}}(\rho) \in L^r(\p, N, \R{Gal}(L/KL_0))$ in $D_{\R{st}}^L(\rho\otimes_{\bZ_p}\bQ_p)$. By the proof of Lemma \ref{lem:3.11}, $M_{\R{st}}(\rho)$ is a free $\cO_H \otimes_{\bZ_p}\cO_{L_0}$-module of rank $d$. Thus, for all $\sigma \in I_{L/K}$, $\R{tr}(\sigma | D_{\R{st}}^L(\rho\otimes_{\bZ_p}\bQ_p)) = \R{tr}(\sigma | M_{\R{st}}(\rho)) \in \cO_H$.

Now, fix a positive integer $n$, and let $B_n  = A_n[\frac{1}{p}]$. $B_{n, \R{red}} \cong \prod_i F_i$ for some finite extensions $F_i/E$. Let $H'/H$ be a finite Galois extension which contains the Galois closures of $F_i$ for all $i$. Similarly as in the proof of Theorem \ref{thm:3.3}, we see that $\cO_{H'}\otimes_{\cO_H}A_n$ is good (as defined in Section \ref{sec:3.1}). Thus, $M_{\R{st}}(\cO_{H'}\otimes_{\cO_H}\rho_n) \in L^r(\p, N, \R{Gal}(L/KL_0))$ is a free $(\cO_{H'}\otimes_{\cO_H}A_n)\otimes_{\bZ_p}\cO_{L_0}$-module of rank $d$ by the proof of Lemma \ref{lem:3.11}. 

Note that we have the surjection $\cO_{H'}\otimes_{\cO_H}\beta_n: \cO_{H'}\otimes_{\cO_H}A_n \twoheadrightarrow \cO_{H'}\otimes_{\cO_H}A/p^n A= \cO_{H'}/p^n\cO_{H'}$. Denote by $T$ the torsion $\cG_K$-representation 
\[
\cO_{H'}/p^n\cO_{H'} \otimes_{\cO_{H'}}(\cO_{H'}\otimes_{\cO_H}\rho) \cong \cO_{H'}/p^n\cO_{H'} \otimes_{\cO_{H'}\otimes_{\cO_H}\beta_n, ~\cO_{H'}\otimes_{\cO_H}A_n}(\cO_{H'}\otimes_{\cO_H}\rho_n).
\]
$T$ has two lifts $j_1$ and $j_2$ corresponding to $\cO_{H'}\otimes_{\cO_H}\rho$ and $\cO_{H'}\otimes_{\cO_H}\rho_n$ respectively, and we obtain $M_{\R{st}, j_1}(T), ~M_{\R{st}, j_2}(T) \in M_{\R{tor}}^{\R{fil}, r}(\p, N, \R{Gal}(L/KL_0))$. Note that $M_{\R{st}, j_1}(T) \cong \cO_{H'}/p^n\cO_{H'}\otimes_{\cO_{H'}}M_{\R{st}}(\cO_{H'}\otimes_{\cO_H}\rho)$ and $M_{\R{st}, j_2}(T) \cong (\cO_{H'}\otimes_{\cO_H}A_n)/\R{ker}(\cO_{H'}\otimes_{\cO_H}\beta_n)\otimes_{\cO_{H'}\otimes_{\cO_H}A_n}M_{\R{st}}(\cO_{H'}\otimes_{\cO_H}\rho_n)$. Thus, $M_{\R{st}, j_1}(T)$ and $M_{\R{st}, j_2}(T)$ are free over $\cO_{H'}/p^n\cO_{H'} \otimes_{\bZ_p} \cO_{L_0}$ of rank $d$, for which we fix a choice of bases.  Let $\sigma \in I_{L/K}$, and for $i= 1, 2$, let $C_i$ be the $d \times d$ matrix with coefficients in $\cO_{H'}/p^n\cO_{H'} \otimes_{\bZ_p} \cO_{L_0}$ which represents the $\sigma$-action on $M_{\R{st}, j_i}(T)$ with respect to the chosen bases. By Corollary \ref{cor:2.8} and Proposition \ref{prop:2.9}, there exist $I_{L/K}$-equivariant $\cO_{H'}\otimes_{\bZ_p}\cO_{L_0}$-module morphisms $g_1: M_{\R{st}, j_1}(T) \arr M_{\R{st}, j_2}(T)$ and $g_2: M_{\R{st}, j_2}(T) \arr M_{\R{st}, j_1}(T)$ such that $g_1 \circ g_2 = p^{c''}\R{Id}|_{ M_{\R{st}, j_2}(T)}$ and $g_2 \circ g_1 = p^{c''}\R{Id}|_{ M_{\R{st}, j_1}(T)}$ where $c''$ is a constant depending only on the Eisenstein polynomial for $L/L_0$ and $r$. For $i = 1, 2$, let $D_i$ be the $d \times d$ matrix with coefficients in $\cO_{H'}/p^n\cO_{H'} \otimes_{\bZ_p} \cO_{L_0}$ representing $g_i$. Then $D_1 D_2 = D_2 D_1 = p^{c''} \R{Id}$ and $C_2D_1 = D_1C_1$. Thus,
\[
\R{tr}(C_2D_1D_2) = \R{tr}(D_1C_1D_2) = \R{tr}(D_2D_1C_1),
\]
i.e., $p^{c''}\R{tr}(C_1) = p^{c''}\R{tr}(C_2)$ in $\cO_{H'}/p^n\cO_{H'}$. Since $\R{tr}(\sigma | M_{\R{st}}(\cO_{H'}\otimes_{\cO_H}\rho_n)) = \R{tr}(\tau(\sigma)) \in \cO_E$ and $\R{tr}(\sigma | M_{\R{st}}(\cO_{H'}\otimes_{\cO_H}\rho)) = \R{tr}(\sigma | M_{\R{st}}(\rho)) \in \cO_H$, we have
\[
\R{tr}(\sigma|M_{\R{st}}(\tilde{\rho})) - \R{tr}(\tau(\sigma)) \in p^{n-c''}\cO_H.
\]  
Since this holds for all positive integers $n$, we have $\R{tr}(\sigma|M_{\R{st}}(\tilde{\rho})) = \R{tr}(\tau(\sigma))$.
\end{proof}

\section{Galois Deformation Ring} \label{sec:4}

We now construct the quotient of the universal deformation ring which corresponds to the locus of potentially semi-stable representations of a given Hodge-Tate type and Galois type. Let $E/\bQ_p$ be a finite extension with residue field $\bF$. Denote by $\cC$ the category of topological local $\cO_E$-algebras $A$ satisfying the following conditions:

\begin{itemize}
\item The natural map $\cO_E \rightarrow A/\fm_A$ is surjective.
\item The map from $A$ to the projective limit of its discrete artinian quotients is a topological isomorphism.	
\end{itemize}

Note that the first condition implies $\bF$ is also the residue field of $A$. The second condition is equivalent to the condition that $A$ is complete and its topology can be given by a collection of open ideals $\fa$ for which $A/\fa$ is artinian. Morphisms in $\cC$ are continuous $\cO_E$-algebra homomorphism. 

\begin{prop} \emph{(\cite[Proposition 2.4]{smit})} \label{prop:4.1}
Suppose $A$ is a noetherian ring in $\cC$. Then the topology on $A$ is equal to the $\fm_A$-adic topology, and $A$ is $\fm_A$-adically complete. Furthermore, every $\cO_E$-algebra homomorphism $A \rightarrow A'$ with $A'$ in $\cC$ is continuous.	
\end{prop}

Let $V_0$ be a continuous $\bF$-representation of $\cG_K$ having rank $d$. For $A \in \cC$, a \textit{deformation} of $V_0$ in $A$ is an isomorphism class of continuous $A$-representations $V$ of $\cG_K$ satisfying $\bF\otimes_A V \cong V_0$ as $\bF[\cG_K]$-modules. We denote by $\R{Def}(V_0, A)$ the set of such deformations. A morphism $A \rightarrow A'$ in $\cC$ induces a map $f_*: \R{Def}(V_0, A) \rightarrow \R{Def}(V_0, A')$ sending the class of a representation $V$ over $A$ to the class of $A'\otimes_{f, A} V$. Assume $V_0$ is absolutely irreducible. Then, the following is proved in \cite{smit}.

\begin{prop} \label{prop:4.2} \emph{(cf. \cite[Theorem 2.3]{smit})}
There exists a universal deformation ring $R \in \cC$ and a deformation $V_R \in \R{Def}(V_0, R)$ such that for all rings $A \in \cC$, we have a bijection 
\begin{equation}
\R{Hom}_{\cC}(R, A) \stackrel{\cong}{\rightarrow} \R{Def}(V_0, A)
\end{equation}
given by $f \mapsto f_*(V_R)$. The ring $R$ is noetherian if and only if $\R{dim}_{\bF}H^1(\cG_K, \R{End}_{\bF}(V_0))$ is finite. 
\end{prop}

Note that if $K/\bQ_p$ is not finite, then $R$ is not necessarily noetherian in general.

We fix a Hodge-Tate type $\bm{v}$ and Galois type $\tau$, and let $L/K$ be a finite Galois extension over which $\tau$ becomes trivial. Let $\cC^0$ be the full subcategory of $\cC$ consisting of artinian rings. Abusing the notation, we write $V \in \R{Def}(V_0, A)$ for a continuous $A$-representation $V$ to mean that $\bF\otimes_A V \cong V_0$. For $A \in \cC^0$ and a $\cG_K$-representation $V_A \in \R{Def}(V_0, A)$, we say $V_A$ is $\textit{potentially semi-stable of type}$ $(\bm{v}, \tau)$ if there exist a finite flat $\cO_E$-algebra $B$, a surjection $g: B \arr A$ of $\cO_E$-algebras, and a continuous $B$-representation $V_B$ of $\cG_K$  such that $V_B \otimes_{\bZ_p} \bQ_p$ is potentially semi-stable having Hodge-Tate type $\bm{v}$ and Galois type $\tau$, and $A \otimes_{g, B} V_B \cong V_A$ as $A[\cG_K]$-modules. For $A \in \cC$, denote by $S_{\bm{v}, \tau}(A)$ the subset of $\R{Def}(V_0, A)$ consisting of the isomorphism classes of representations $V_A$ such that $A/\fa\otimes_A V_A$ is potentially semi-stable of type $(\bm{v}, \tau)$ for all open ideals $\fa \subsetneq A$. 

\begin{prop} \label{prop:4.3}
For any $\cC$-morphism $f: A \rightarrow A'$, we have $f_*(S_{\bm{v}, \tau}(A)) \subset S_{\bm{v}, \tau}(A')$. There exists a closed ideal $\fa_{\bm{v}, \tau}$ of the universal deformation ring $R$ such that the map $(4.1)$ induces a bijection $\R{Hom}_{\cC}(R/\fa_{\bm{v}, \tau}, A) \stackrel{\cong}{\rightarrow} S_{\bm{v}, \tau}(A)$.
\end{prop}

\begin{proof}
We check the conditions in \cite[Section 6]{smit}. Let $f: A \hookrightarrow A'$ be an inclusion of artinian rings in $\cC$, and let $V_A \in \R{Def}(V_0, A)$ be a representation. We first claim that $V_A \in S_{\bm{v}, \tau}(A)$ if and only if $V_{A'} \coloneqq A'\otimes_{f,A} V_A \in S_{\bm{v}, \tau}(A')$. Suppose that $V_A \in S_{\bm{v}, \tau}(A)$. Then there exist a finite flat $\cO_E$-algebra $B$, a surjection $g: B \arr A$, and a $B$-representation $V_B$ such that $V_B \otimes_{\bZ_p} \bQ_p$ is potentially semi-stable having Hodge-Tate type $\bm{v}$ and Galois type $\tau$, and $A \otimes_{g, B} V_B \cong V_A$. There exists a surjection $f': A[x_1, \ldots, x_n] \twoheadrightarrow A'$ of $\cO_E$-algebras extending $f$ such that $f'(x_i) \in \fm_{A'}$ for each $i$. Let $I_{m, A} \subset A[x_1, \ldots, x_n]$ denote the ideal generated by the $m$-th degree homogeneous polynomials with coefficients in $A$. Since $A'$ is artinian, $f'(I_{m, A}) = 0$ for a sufficiently large $m$, and $f'$ induces a surjection $A[x_1, \ldots, x_n]/I_{m, A} \twoheadrightarrow A'$ for such $m$. Thus, we have surjective homomorphisms of $\cO_E$-algebras
\[
g': B' \coloneqq B[x_1, \ldots, x_n]/ I_{m, B} \twoheadrightarrow A[x_1, \ldots, x_n]/I_{m, A} \twoheadrightarrow A'.
\]   
Note that $B'$ is a finite flat $\cO_E$-algebra. Let $V_{B'} = B'\otimes_B V_B$. Then $V_{B'}\otimes_{\bZ_p}\bQ_p$ is semi-stable over $L$. By Lemma \ref{lem:3.2} and \ref{lem:3.14}, it has Hodge-Tate type $\bm{v}$ and Galois type $\tau$. $A'\otimes_{g', B'}V_{B'} \cong V_{A'}$, so $V_{A'} \in S_{\bm{v}, \tau}(A')$.

Conversely, suppose $V_{A'} \in S_{\bm{v}, \tau}(A')$. Then there exist a finite flat $\cO_E$-algebra $B'$, a surjection $g': B' \arr A'$, and a $B'$-representation $V_{B'}$ such that $V_{B'} \otimes_{\bZ_p} \bQ_p$ is potentially semi-stable having Hodge-Tate type $\bm{v}$ and Galois type $\tau$, and we have an isomorphism $h: A' \otimes_{g', B'} V_{B'} \stackrel{\cong}{\longrightarrow} V_{A'}$ of $A'$-representations. Since $\cO_E$ is henselian, $B'$ is a finite product of local rings flat over $\cO_E$. By Lemma \ref{lem:3.2} and Lemma \ref{lem:3.14}, we can take $B'$ to be a local ring, since $A'$ is local. Thus, we can lift the isomorphism $h$ to an isomorphism of $B'$-modules $h_1: V_{B'} \stackrel{\cong}{\longrightarrow} V_{B'}$ such that the composite
\[
A'\otimes_{g', B'} V_{B'} \stackrel{\mbox{id}\otimes h_1^{-1}}{\longrightarrow} A'\otimes_{g', B'} V_{B'} \stackrel{h}{\longrightarrow} V_{A'}
\] 
is the identity map of $A'$-modules. 

Let $B$ be the kernel of the composite of morphisms $B' \stackrel{g'}{\rightarrow}A' \rightarrow A'/f(A)$. Then $B$ is a finite flat $\cO_E$-algebra, and we have the surjection $g: B \twoheadrightarrow A$ of $\cO_E$-algebras induced from $g'$. Let $V_B$ be the kernel of the following composite of morphisms
\[
V_{B'} \stackrel{h_1^{-1}}{\longrightarrow} V_{B'} \rightarrow A' \otimes_{g', B'} V_{B'} \stackrel{h}{\longrightarrow} V_{A'} \rightarrow A'/f(A)\otimes_{A'}V_{A'}.
\]
Then $V_B$ is a continous $B$-representation of $\cG_K$ such that $B'\otimes_B V_B \cong V_{B'}$ and $A\otimes_{g, B}V_B \cong V_A$, since $V_{A'} = A'\otimes_{f,A} V_A $. By the main theorem for semi-stable representations in \cite{liu-fontaineconjecture}, $V_B \otimes_{\bZ_p}\bQ_p$ is semi-stable over $L$. It has Hodge-Tate type $\bm{v}$ and Galois type $\tau$ by Lemma \ref{lem:3.8} and \ref{lem:3.14}, and therefore $V_A \in S_{\bm{v}, \tau}(A)$.

Now, for $A \in \cC$ and a representation $V_A \in \R{Def}(V_0, A)$, suppose $\fa_1, \fa_2 \subsetneq A$ are open ideals such that $A/\fa_i\otimes_A V_A \in S_{\bm{v}, \tau}(A/\fa_i)$ for $i = 1, 2$. We claim that $A/(\fa_1\cap\fa_2)\otimes_A V_A \in S_{\bm{v}, \tau}(A/(\fa_1\cap\fa_2))$. There exist a finite flat $\cO_E$-algebra $B_i$, a surjection $g_i: B_i \twoheadrightarrow A/\fa_i$, and a $B_i$-representation $V_{B_i}$ such that $V_{B_i}\otimes_{\bZ_p}\bQ_p$ is potentially semi-stable having Hodge-Tate type $\bm{v}$ and Galois type $\tau$, and that $A/\fa_i\otimes_{g_i, B_i}V_{B_i} \cong A/\fa_i\otimes_A V_A$. Let $V_{B_1 \times B_2}$ be the $(B_1 \times B_2)$-representation corresponding to $V_{B_1}\oplus V_{B_2}$. Note that $V_{B_1 \times B_2}\otimes_{\bZ_p}\bQ_p$ is potentially semi-stable having Hodge-Tate type $\bm{v}$ and Galois type $\tau$. Consider the natural inclusion $A/(\fa_1\cap\fa_2) \subset A/\fa_1 \times A/\fa_2$. Let $B$ be the kernel of the composite of morphisms 
\[
B_1 \times B_2 \stackrel{g_1\times g_2}{\longrightarrow} A/\fa_1 \times A/\fa_2 \rightarrow (A/\fa_1 \times A/\fa_2)/(A/(\fa_1\cap\fa_2)). 
\]   
Then $B$ is a finite flat $\cO_E$-algebra, and we have the surjection $g: B \rightarrow A/(\fa_1\cap\fa_2)$ induced from $g_1 \times g_2$. Let $V_B$ be the kernel of the composite of morphisms
\[
V_{B_1 \times B_2} \rightarrow (A/\fa_1 \times A/\fa_2)\otimes_{g_1\times g_2, B_1\times B_2}V_{B_1 \times B_2} \cong (A/\fa_1 \times A/\fa_2)\otimes_A V_A
\]
and
\[
(A/\fa_1 \times A/\fa_2)\otimes_A V_A \rightarrow (A/\fa_1 \times A/\fa_2)/(A/(\fa_1\cap\fa_2)) \otimes_A V_A.
\]
Then $V_B$ is a continuous $B$-representation of $\cG_K$ such that $(B_1 \times B_2) \otimes_B V_B \cong V_{B_1\times B_2}$ and $A/(\fa_1\cap\fa_2)\otimes_{g, B}V_B \cong A/(\fa_1\cap\fa_2) \otimes_A V_A$. By the main theorem for semi-stable representations in \cite{liu-fontaineconjecture}, $V_B \otimes_{\bZ_p}\bQ_p$ is semi-stable over $L$. It has Hodge-Tate type $\bm{v}$ and Galois type $\tau$ by Lemma \ref{lem:3.8} and \ref{lem:3.14}. Thus, $A/(\fa_1\cap\fa_2)\otimes_A V_A \in S_{\bm{v}, \tau}(A/(\fa_1\cap\fa_2))$.

The result then follows by \cite[Proposition 6.1]{smit}.
\end{proof}

Finally, we prove the main theorem.

\begin{thm} \label{thm:4.4}
Let $A$ be a finite flat $\cO_E$-algebra, and let $f: R \rightarrow A$ be a continuous $\cO_E$-algebra homomorphism (where we equip $A$ with the $(p)$-adic topology). Then the induced representation $A[\frac{1}{p}]\otimes_{f, R}V_R$ is potentially semi-stable of Hodge-Tate type $\bm{v}$ and Galois type $\tau$ if and only if $f$ factors through the quotient $R/\fa_{\bm{v}, \tau}$.	
\end{thm}

\begin{proof}
Let $A_1 = f(R) \subset A$. Then $A_1$ is a finite flat $\cO_E$-algebra and local. We equip $A_1$ with the $(p)$-adic topology. Then $A_1 \in \cC$, and the map $f: A \rightarrow A_1$ is a morphism in $\cC$. Let $V_{A_1} = A_1 \otimes_{f, R} V_R$ and $V_A = A\otimes_{f, R}V_R \cong A \otimes_{A_1}  V_{A_1}$.  

Suppose that $V_A\otimes_{\bZ_p} \bQ_p$ is potentially semi-stable of Hodge-Tate type $\bm{v}$ and Galois type $\tau$. By the main theorem for semi-stable representations in \cite{liu-fontaineconjecture}, $V_{A_1}\otimes_{\bZ_p}\bQ_p$ is semi-stable over $L$. By Lemma \ref{lem:3.8} and \ref{lem:3.14}, $V_{A_1}\otimes_{\bZ_p}\bQ_p$ has Hodge-Tate type $\bm{v}$ and Galois type $\tau$. Thus, $V_{A_1} \in S_{\bm{v}, \tau}(A_1)$, and $f$ factors through $R/\fa_{\bm{v}, \tau}$ by Proposition \ref{prop:4.3}.

Conversely, suppose $f$ factors through $R/\fa_{\bm{v}, \tau}$. Then $V_{A_1} \in S_{\bm{v}, \tau}(A_1)$ by Proposition \ref{prop:4.3}, so $A_1/p^n\otimes_{A_1} V_{A_1}$ is potentially semi-stable of type $(\bm{v}, \tau)$ for each $n \geq 1$. By the main theorem for semi-stable representations in \cite{liu-fontaineconjecture}, $V_{A_1}\otimes_{\bZ_p} \bQ_p$ is semi-stable over $L$. And by Theorem \ref{thm:3.3} and \ref{thm:3.17}, $V_{A_1}\otimes_{\bZ_p} \bQ_p$ has Hodge-Tate type $\bm{v}$ and Galois type $\tau$. Thus, $V_A \otimes_{\bZ_p}\bQ_p$ is potentially semi-stable of Hodge-Tate type $\bm{v}$ and Galois type $\tau$. 
\end{proof}

\bibliographystyle{amsalpha}
\bibliography{library}
	
\end{document}